\documentclass{TPmod}
\usepackage{url}
\usepackage{setspace}
\usepackage{scrextend}

\usepackage{amsthm}
\usepackage{amssymb}
\usepackage[capitalize]{cleveref}
\usepackage{mathtools}

\DeclareMathOperator{\Aut}{Aut}
\DeclareMathOperator{\Auteq}{Aut}
\DeclareMathOperator{\Ext}{Ext}

\newcommand{\Spec}{\mathrm{Spec}}
\newcommand{\Sym}{\mathrm{Sym}}

\newcommand{\st}{\mathrm{st}}
\newcommand{\ct}{\mathrm{ct}}

\newcommand{\scrM}{\EuScript{M}}

\DeclareMathOperator{\Diff}{Diff}

\DeclareMathOperator{\Symp}{Symp}
\DeclareMathOperator{\Stab}{Stab}
\def\P{\mathcal{P}}
\newcommand{\NS}{\operatorname{NS}}

\def\bP{\mathbb{P}}
\def\bA{\mathbb{A}}

\def\bZ{\mathbb{Z}}

\def\bc{\mathsf{bc}}

\def\Dpinf{\EuD}
\def\Db{\mathrm{D}}
\def\Dbdg{\EuD}
\def\val{\mathrm{val}}

\def\pol{\Delta}
\def\fk{\mathsf{k}}
\def\fK{\mathsf{K}}
\def\Ch{\mathsf{Ch}}
\def\amp{0}
\def\qa{Q1}
\def\qb{Q2}
\def\qc{Q3}
\def\qd{Q4}
\def\qe{Q5}
\def\qf{Q6}
\def\qg{Q7}
\def\cM{\mathcal{M}}
\def\cL{\mathcal{L}}
\def\bL{\mathbb{N}}
\def\Po{\mathrm{P}^\circ}

\renewcommand{\rk}{\operatorname{rk}}
\renewcommand{\cong}{\simeq}

\newcommand{\simxleftrightarrow}[1]{\underset{\mathrel{\raisebox{2pt}{$\sim$}}}{\xleftrightarrow{#1}}}

\begin{document}

\title{Symplectic topology of $K3$ surfaces via mirror symmetry}

\author[Sheridan and Smith]{Nick Sheridan and Ivan Smith}

\address{Nick Sheridan, School of Mathematics, University of Edinburgh, Edinburgh EH9 3FD, U.K.}
\address{Ivan Smith, Centre for Mathematical Sciences, University of Cambridge, Wilberforce Road, Cambridge CB3 0WB, U.K.}

\begin{abstract} {\sc Abstract:} We study the symplectic topology of certain $K3$ surfaces (including the ``mirror quartic'' and ``mirror double plane''), 
equipped with certain K\"ahler forms.  In particular, we prove that the symplectic Torelli group may be infinitely generated, and derive new constraints on Lagrangian tori. The key input, via homological mirror symmetry, is a result of Bayer and Bridgeland on the autoequivalence group of the derived category of an algebraic $K3$ surface of Picard rank one.
\end{abstract}
\maketitle

\section{Introduction}

Let $(X,\omega)$ be a simply-connected closed symplectic 4-manifold. 
Elucidating the structure of the symplectic mapping class group $\pi_0\Symp(X,\omega)$, and of its natural map to $\pi_0\Diff(X)$, are central aims in four-dimensional topology. 
This paper brings to bear insight from homological mirror symmetry, in the specific case of (certain symplectic forms on) $K3$ surfaces.

\subsection{Context and questions}

Many aspects of $\pi_0\Diff(X)$ seem \emph{a priori} out of reach, since we have no knowledge of $\pi_0\Diff_{\ct}(\R^4)$ -- there is therefore no smooth closed four-manifold for which the smooth mapping class group is known --  but in the symplectic setting we have Gromov's theorem $\Symp_{\ct}(\C^2,\omega_{\st}) \simeq \{1\}$, via his spectacular determination that $\Symp(\C\bP^2,\omega_{\st}) \simeq \bP U(3)$ and $\Symp(\bP^1 \times \bP^1, \omega \oplus \omega) \simeq (SO(3) \times SO(3)) \rtimes \Z/2$. 

Away from the setting of rational or ruled surfaces, much less is known. 
We restrict our attention to the symplectic mapping class group and its variants:
\begin{align*}
G(X,\omega) & \coloneqq \pi_0\Symp(X,\omega)\quad\text{(symplectic mapping class group)} \\
I(X,\omega) & \coloneqq \ker\, [G(X,\omega) \rightarrow \aut(H^*(X;\mathbb{Z}))]\quad \text{(symplectic Torelli group)} \\
K(X,\omega) & \coloneqq \ker \, [G(X,\omega) \to \pi_0\Diff(X)] \quad\text{(smoothly trivial symplectic mapping class group).}
\end{align*}
Results of Seidel \cite{Seidel:knotted} and Tonkonog \cite{Tonkonog} show that $K(X,\omega)$ is infinite for ``most'' even-dimensional hypersurfaces in projective space, and Seidel \cite{Seidel:LecturesDehn} has shown that for certain symplectic $K3$ surfaces $K(X,\omega)$ contains a subgroup isomorphic to the pure braid group $PB_m$ for $m \le 15$.

However, very little is known about the structure of $G(X,\omega)$ itself.
For instance, the following questions\footnote{These attributions are folkloric.} remain open in almost all cases:
\begin{enumerate}
\item[(\qa)] (Seidel) Is $G(X,\omega)$ finitely generated?
\item[(\qb)] (Fukaya) Does $G(X,\omega)$ have finitely many orbits on the set of Lagrangian spheres in $X$?
\item[(\qc)] (Donaldson) Is $I(X,\omega)$ generated by squared Dehn twists in Lagrangian spheres?
\end{enumerate}
Contrast with the (symplectic) mapping class group $\Gamma_g$ of a closed oriented surface $\Sigma_g$ of genus $g$, which is finitely presented, generated by Dehn twists, and acts with finitely many orbits on the infinite set of isotopy classes of simple closed curves on $\Sigma_g$, or the classical Torelli group, which is infinitely generated when $g=2$ and finitely generated when $g>2$.  

One also has long-standing questions concerning the structure of Lagrangian submanifolds in a complex projective manifold, for example:
\begin{enumerate}
\item[(\qd)] (Donaldson \cite[Question 4]{Donaldson2000}) Do all Lagrangian spheres in a complex projective manifold arise from the vanishing cycles of deformations to singular varieties? 
\end{enumerate}
A negative answer for certain rigid Calabi-Yau threefolds was very recently given in \cite{Greer}. 
More specific to a symplectic Calabi--Yau surface $(X,\omega)$, we have: 
\begin{enumerate}
\item[(\qe)] \label{it:torusnonvan} (Seidel) Does every Lagrangian torus $T^2 \subset X$ of Maslov class zero represent a non-zero class in $H_2(X,\Z)$?
\end{enumerate}

The corresponding result for the four-dimensional torus (with the standard flat K\"ahler form)  has a positive answer, see \cite[Corollary 9.2]{AbouzaidSmith}.

\subsection{Sample results}

We are able to give partial answers to some of these questions, in particular we prove:

\begin{thm}\label{Thm:Torelli}
There is a symplectic $K3$ surface $(X,\omega)$ for which the symplectic Torelli group $I(X,\omega)$ surjects onto a free group of (countably) infinite rank, and in particular is infinitely generated.
\end{thm}

To be more explicit, the result holds for ambient-irrational K\"ahler forms on either of the following two $K3$ surfaces:
\begin{enumerate}
\item $X$ is the ``mirror quartic'', i.e. the crepant resolution of a quotient of the Fermat quartic hypersurface in $\C\bP^3$ by $(\Z/4)^2$;
\item $X$ is the ``mirror double plane'', i.e. the crepant resolution of a quotient of the Fermat sextic hypersurface in $\C\bP(1,1,1,3)$ by the group $\Z/6 \times \Z/2$.
\end{enumerate}
In both cases, $X \subset Y$ is the proper transform of an anticanonical hypersurface $\bar{X} \subset \bar{Y}$ in the toric variety $\bar{Y}$ associated to an appropriate reflexive simplex $\pol$, where $Y \to \bar{Y}$ is a toric resolution of singularities that is chosen in such a way that $X$ is a crepant resolution of $\bar{X}$. 
A K\"ahler form on $X$ is ``ambient'' if it is restricted from $Y$, and ``ambient-irrational'' if it is ambient and $[\omega]^{\perp} \cap H^2(X,\Z)$ is as small as possible amongst ambient K\"ahler forms.

\begin{rmk} \label{rmk:level_of_generality}
Although we only establish Theorem \ref{Thm:Torelli} in these rather specific cases, the underlying strategy -- which extracts information on symplectic mapping class groups from homological mirror symmetry -- is of interest in itself, and has much broader potential.  As we elaborate in Section \ref{Sec:Outlines} below, this paper essentially shows that Theorem \ref{Thm:Torelli} holds whenever one has proved homological mirror symmetry for $(X,\omega)$, and has proved Bridgeland's conjecture on the autoequivalences of the derived category of its mirror algebraic $K3$ surface $X^{\circ}$.  In particular, it seems likely that ongoing developments will give the same conclusion in much greater generality, cf. for instance Remark \ref{rmk:cubic_fourfold}.  \end{rmk}

\begin{rmk}
\label{rmk:space_of_symp}
In fact we prove a stronger result: the surjection from $I(X,\omega)$ onto an infinite-rank free group remains surjective when restricted to the subgroup $K(X,\omega) \le I(X,\omega)$ (cf. Corollary \ref{cor:smoothly_trivial_inf_generated}).  This has the following consequence, suggested to us by Dietmar Salamon.  Following  \cite{Kronheimer:families}, if $\Diff_0(X)$ denotes the group of diffeomorphisms which are smoothly isotopic to the identity, then $\Diff_0(X)$ acts on 
\[
\Omega = \{a\in \Omega^2(X) \, | \, a \ \textrm{is symplectic and cohomologous to} \ \omega \}
\]
via the map $f \mapsto (f^{-1})^*\omega$. Moser's theorem shows the action is transitive on the connected component containing $\omega$, with stabiliser $\Symp_0(X) = \Symp(X,\omega) \cap \Diff_0(X)$. The long exact sequence of homotopy groups for the associated Serre fibration yields a surjective homomorphism 
\[\pi_1(\Omega,\omega) \twoheadrightarrow K(X,\omega).\]
It follows that the fundamental group of the space of symplectic forms  on $X$ isotopic to $\omega$ surjects onto an infinite-rank free group. 
See also Lemma \ref{lem:symp_forms}.
\end{rmk}

Seidel's question (\qe) above, in the case of a symplectic $K3$ surface $(X,\omega)$, was motivated by the following beautiful line of thought, explained to us by Seidel: suppose that $L \subset (X,\omega)$ is a Maslov-zero Lagrangian torus with vanishing homology class, and $X^\circ$ is a $K3$ mirror to $(X,\omega)$. 
There exist non-commutative deformations of $\Db(X^\circ)$ which destroy all point-like objects (i.e., for which all point-like objects become obstructed); therefore the corresponding symplectic deformations of $(X,\omega)$ should destroy the point-like object $L$. 
However if $[L] = 0$, then we can deform $L$ to remain Lagrangian under any deformation of the symplectic form by a Moser-type argument: a contradiction, so such $L$ could not have existed.
We use a variation on this idea to prove:

\begin{thm} \label{Thm:Lag_torus}
Let $X$ be a mirror quartic or mirror double plane, and $\omega$ any ambient K\"ahler form on $X$. Then every Maslov-zero Lagrangian torus $L\subset (X,\omega)$ has non-trivial homology class.
\end{thm}

The Maslov class hypothesis is obviously necessary, since there are Lagrangian tori in a Darboux chart. Under stronger hypotheses one can say more; in particular, if the ambient K\"ahler form is ambient-irrational, then Maslov-zero Lagrangian tori are homologically primitive.

Some of our results do not concern the groups $G(X,\omega)$ and $I(X,\omega)$ themselves, but rather their homological algebraic cousin $\Auteq \Dpinf\EuF(X)$, where $\Dpinf\EuF(X)$ is the split-closed derived Fukaya category.  For instance, we show that for the particular symplectic $K3$ surfaces considered in Theorem \ref{Thm:Torelli},  the corresponding Torelli-type group $\ker\,[\Auteq \Dpinf\EuF(X) \to \Aut HH_*(\Dpinf\EuF(X))]$ is indeed generated by squared Dehn twists, whilst $\Auteq \Dpinf\EuF(X)$ itself is finitely presented, but not generated by Dehn twists, cf. Corollary \ref{cor:not_generated_by_dehn_twists}.

In the same vein, we give positive answers to weakened versions of questions (\qb) and (\qd) in certain circumstances:

\begin{thm}
\label{Thm:spheres}
Let $X$ be a mirror quartic or mirror double plane, and $\omega$ an ambient-irrational K\"ahler form on $X$. Then:
\begin{enumerate}
\item Every Lagrangian sphere is Fukaya-isomorphic to a vanishing cycle.
\item $G(X,\omega)$ acts transitively on the set of Fukaya-isomorphism classes of Lagrangian spheres.
\end{enumerate}
\end{thm}  

\begin{rmk}
In the body of this paper we will only consider four-dimensional symplectic Calabi--Yau manifolds, for which a complete construction of the Fukaya category can be carried out using classical pseudoholomorphic curve theory. It is interesting to note that our results would have implications in higher dimensions, once the relevant foundational theory is established. For instance, let $(X,\omega)$ be a symplectic $K3$ surface as in Theorem \ref{Thm:Torelli}. Consider the product $Z = (X\times S^2, \omega \oplus \omega_{FS})$. Conjecturally, 
\[
\Dpinf\EuF(Z) \simeq \Dpinf\EuF(X) \oplus \Dpinf\EuF(X)\]
splits as a direct product of two copies of $\Dpinf\EuF(X)$, compare to \cite[Example 3.23]{Seidel:Flux}.  Given this, one infers that the symplectic Torelli group $I(Z)$ surjects onto an infinite-rank free group. By contrast, for simply-connected manifolds of dimension $\geq 5$, Sullivan  showed \cite[Theorem 10.3]{Sullivan} that the differentiable Torelli group $\ker \, [\pi_0\Diff(Z) \to \Aut H^*(Z)]$ is commensurable with an arithmetic group, in particular is finitely presented. Thus, such a stabilised version of Theorem \ref{Thm:Torelli} would be intrinsically symplectic in nature. (Our lack of knowledge of four-dimensional smooth mapping class groups makes it hard to draw precisely the same contrast in the setting of Theorem \ref{Thm:Torelli} itself.)  
\end{rmk}

\subsection{Outlines\label{Sec:Outlines}}

The results above are obtained by combining two main ingredients:
\begin{itemize}
\item the proof of homological mirror symmetry for Greene--Plesser mirrors, cf. \cite{SS};
\item Bayer and Bridgeland's proof \cite{BB} of Bridgeland's conjecture \cite{Bridgeland-K3} on the group of autoequivalences of the derived category $\Db(X^{\circ})$ of a $K3$ surface $X^{\circ}$, in the case that $X^\circ$ has Picard rank one.
\end{itemize}

\begin{rmk}
An important point is that Bayer and Bridgeland prove Bridgeland's conjecture for Picard rank one $K3$ surfaces, but the higher-rank case remains open. 
In particular, their proof does not apply to the mirror of the quartic surface (which has Picard rank $19$). 
That is why we cannot use Seidel's proof of homological mirror symmetry in that case  \cite{Seidel:HMSquartic} as input for our results, but must instead use the more general results proved in \cite{SS}.  
\end{rmk}

\begin{rmk}
Another slight  delicacy is that the mirror to a symplectic $K3$ surface $(X,\omega)$, in the sense of homological mirror symmetry, is an algebraic $K3$ surface over the Novikov field $\Lambda$, whilst the work of Bayer and Bridgeland, at least as written, is for $K3$ surfaces over $\C$. We circumvent this using standard ideas around the ``Lefschetz principle'', i.e. the fact that any algebraic $K3$ surface over a field of characteristic zero is in fact defined over a finitely generated extension field of the rationals, and such fields admit embeddings in $\C$. 
\end{rmk}

We now explain how we combine these two ingredients to obtain information about symplectic mapping class groups. 
We study $G(X,\omega)$ via its action on the Fukaya category. 
It turns out always to act by ``Calabi--Yau'' autoequivalences, and the action is only well-defined up to the action of even shifts, so that we have a homomorphism
\[ G(X,\omega) \to \Auteq_{CY} \Dpinf\EuF(X,\omega)/[2].\]
If $X^\circ$ is homologically mirror to $(X,\omega)$, then the autoequivalence group of $\Dpinf\EuF(X,\omega)$ can be identified with the autoequivalence group of the derived category $\Db(X^\circ)$. 
The Bayer--Bridgeland theorem identifies the subgroup of Calabi--Yau autoequivalences of $\Db(X^\circ)$, modulo even shifts, with $\pi_1(\cM)$, where $\cM = \cM_{\text{K\"ah}}(X^\circ)$ is a version of the ``K\"ahler moduli space'' of $X^\circ$ defined via Bridgeland's stability conditions.

We identify $\cM_{\text{K\"ah}}(X^\circ)$ with the ``complex moduli space'' $\cM_{\text{cpx}}(X,\omega)$, and construct a symplectic monodromy homomorphism $\pi_1(\cM_{\text{cpx}}(X,\omega)) \to G(X,\omega)$. 
We show that the composition
\begin{equation}
\label{eqn:compid1} 
\pi_1(\cM) \to G(X,\omega) \to \Auteq_{CY} \Dpinf\EuF(X,\omega)/[2] \cong \Auteq_{CY} \Db(X^\circ)/[2] \cong \pi_1(\cM) 
\end{equation}
is an isomorphism.
Considering the respective Torelli subgroups, we have a similar composition
\begin{equation}
\label{eqn:compid2} 
\pi_1\left(\tilde{\cM}\right) \to I(X,\omega) \to \Auteq^0 \Dpinf\EuF(X,\omega)/[2] \cong \Auteq^0 \Db(X^\circ)/[2] \cong \pi_1\left(\tilde{\cM}\right)
\end{equation}
which is also an isomorphism, where $\tilde{\cM}$ is a certain cover of $\cM$ (denoted $\Omega^+_0(X,\omega)$, or alternatively $\Omega^+_0(X^\circ)$, in the body of the paper).

It follows, in particular, that $I(X,\omega)$ surjects onto $\pi_1(\tilde{\cM})$. 
Theorem \ref{Thm:Torelli} then follows from the fact that $\tilde{\cM}$ is the complement of a countably infinite discrete set of points in the upper half plane, and in particular its fundamental group is a free group of countably infinite rank.

\subsection{Further questions}

The fact that the compositions \eqref{eqn:compid1}, \eqref{eqn:compid2} are isomorphisms implies that
\begin{align}
\label{eqn:GZ1} G(X,\omega) & \cong  Z(X,\omega)\rtimes\pi_1(\cM)  \quad \text{and}\\
\label{eqn:GZ2} I(X,\omega) & \cong  Z(X,\omega) \rtimes \pi_1(\tilde{\cM}) ,
\end{align}
where
\[Z(X,\omega) \coloneqq  \ker \, \left[G(X,\omega) \to \Auteq\Dpinf\EuF(X,\omega)/[2]\right]\]
denotes the ``Floer-theoretically trivial'' subgroup of the symplectic mapping class group (it is contained in $I(X,\omega)$). 
Thus, we have essentially reduced the problem of computing $G(X,\omega)$ and $I(X,\omega)$ to the problem of computing $Z(X,\omega)$. 
This raises the following:

\begin{itemize}
\item[(\qf)]
Let $(X,\omega)$ be a symplectic $K3$ whose Fukaya category is non-degenerate (i.e., the open-closed map hits the unit). Can $Z(X,\omega)$ be non-trivial?
\end{itemize}

If the answer to (\qf) is negative for ambient-irrational K\"ahler forms on the mirror quartic or mirror double plane, then our results show that the answers to (\qa) and (\qc) are positive in those cases. 

\begin{rmk}
The caveat of non-degeneracy of the Fukaya category is added to (\qf) to avoid a negative answer due to the symplectic manifold ``not having enough Lagrangians to detect symplectomorphisms''.
\end{rmk}

In a similar vein, Theorem \ref{Thm:spheres} reduces questions (\qb) and (\qd) to the following general:

\begin{itemize}
\item[(\qg)] Can there exist Fukaya-isomorphic but non-Hamiltonian-isotopic Lagrangian spheres in a symplectic $K3$?
\end{itemize}

\paragraph{Acknowledgements}  I.S. is indebted to Daniel Huybrechts for  patient explanations of the material of Section \ref{Sec:Lefschetz}.  He is furthermore grateful to Arend Bayer, Tom Bridgeland,  Aurel Page, Oscar Randal-Williams,  Tony Scholl and Richard Thomas for  helpful conversations and correspondence.  Both authors are grateful to the anonymous referees for their queries and suggestions, which have greatly improved the exposition.

N.S. was partially supported by a Sloan Research Fellowship, a Royal Society University Research Fellowship, and by the National Science Foundation through Grant number DMS-1310604 and under agreement number DMS-1128155.
Any opinions, findings and conclusions or recommendations expressed in this material are those of the authors and do not necessarily reflect the views of the National Science Foundation. 
N.S. also acknowledges support from Princeton University and the Institute for Advanced Study.
I.S. was partially supported by a Fellowship from EPSRC.

\section{Preliminaries}

In this section we define some of the moduli spaces that will arise in this paper, following the notation of \cite{BB} closely. The moduli spaces will appear in different ways on the two sides of mirror symmetry, so here we simply give the abstract definitions.

\subsection{Lattices}

In this paper, a \emph{lattice} is a free $\Z$-module of finite rank equipped with a non-degenerate integral symmetric bilinear 
pairing $(\cdot,\cdot)$. 
We will sometimes write $x^2$ for $(x,x)$. We denote the set of $(-2)$-classes in a lattice $L$ by
\begin{equation}
\label{eq:deltaL}
\Delta(L) \coloneqq \{ \delta \in L \, | \, (\delta,\delta) = -2\}.
\end{equation}
If $L$ is a lattice, we denote by $L^-$ the lattice with the same underlying $\Z$-module but the sign of the pairing reversed. 

If $A$ is an abelian group, we denote $L_A \coloneqq L \otimes_\Z A$. 
If $A$ is a $\Z$-algebra, then the pairing on $L$ induces a symmetric bilinear pairing on $L_A$. 
Given $A \subset B$ and a subset $M \subset L_A$, we define 
\[M^\perp_B \coloneqq \{n \in L_B\,|\, (n,m) = 0 \text{ for all $m \in M$}\}.\]
For $m \in L_A$ we will denote $m^\perp_B \coloneqq \{m\}^\perp_B$ (and if $A=B$ we will drop the subscript `$B$').  In particular, for $\delta \in \Delta(L)$, we will write $\delta^{\perp}_{\R}$ and $\delta^{\perp}_{\C}$ for the orthogonals to $\delta$ in the corresponding real or complex vector spaces.

For the rest of the paper, let $L = U^{\oplus 3} \oplus E_8^{\oplus 2}$ be the ``K3 lattice'' (isomorphic to $H^2(X,\Z)$ equipped with the cup product pairing, for any complex K3 surface $X$; here $U$ is the hyperbolic lattice, with matrix $\begin{psmallmatrix} 0 & 1 \\ 1 & 0 \end{psmallmatrix}$ in an appropriate basis).  $L$ has signature $(3,19)$. 

\subsection{Abstract moduli space}

 The \emph{period domain} associated to a lattice $N$ is\footnote{
Our convention is that if $V$ is a complex vector space, then its \emph{projectivization} $\bP(V)$ is the set of complex \emph{lines} in $V$.}
\begin{equation} \label{eqn:period_abstract}
\Omega(N) \ \coloneqq \ \{ \Omega \in \bP(N_{\C}) \, | \, (\Omega,\Omega) = 0, \, (\Omega, \overline{\Omega}) > 0 \}.
\end{equation}
$\Omega(N)$ is an open complex submanifold of a quadric hypersurface in $\bP(N_\C)$. 
The group of lattice automorphisms $\Aut(N)$ acts on $\Omega(N)$ on the left. 

For the rest of this section, let $N$ be a lattice of signature $(2,t)$. 
Observe that $\Omega(N)$ has two connected components, distinguished by the orientation of the positive-definite two-plane spanned by the real and imaginary parts of $\Omega$, and interchanged by complex conjugation. 
We denote them by $\Omega^\pm(N)$.

We define 
\[ \Omega^\pm_\amp(N) \coloneqq \Omega^\pm(N) \setminus \bigcup_{\delta \in \Delta(N)} \bP(\delta^\perp_\C).\]

We define $\Gamma(N) \subset \Aut(N)$ to be the subgroup of isometries acting trivially on the discriminant group $N^*/N$. 
We define $\Gamma^+(N) \subset \Gamma(N)$ to be the subgroup preserving $\Omega^+(N)$. (It may have index one or two, cf. \cite[Proposition 5.6]{Dolgachev}.)

The significance of $\Gamma(N)$ is explained by the following, which is an immediate consequence of \cite[Corollary 1.5.2]{Nikulin} (cf. \cite[Proposition 3.3]{Dolgachev}):

\begin{lem}\label{lem:nikulin}
Suppose that $N$ is a primitive sublattice of the unimodular lattice $M$. Then the image of the natural embedding
\begin{equation}
\label{eqn:gammanperp}
\left\{\sigma \in \Aut(M): \sigma|_{N^\perp} = \id\right\} \hookrightarrow \Aut(N) 
\end{equation}
is $\Gamma(N)$.
\end{lem}

We define the topological stack
\[ \cM(N) \coloneqq [\Omega^+(N)/\Gamma^+(N)],\] 
and similarly $\cM_\amp(N)$. (The curious reader may consult \cite{Noohi} on the general theory of topological stacks and their fundamental groups -- however the only topological stacks in this paper will be global quotient stacks, and they will only appear via their stacky fundamental groups, whose elementary definition we now give.)

The \emph{stacky fundamental group} $\pi_1(\cM(N),[p])$, by definition, consists of all pairs $(\gamma,y)$ where $\gamma \in \Gamma^+(N)$ and $y$ is a homotopy class of paths from $p$ to $\gamma \cdot p$. 
The composition $(\gamma_1,y_1) \cdot (\gamma_2,y_2)$ is defined to be $(\gamma_1 \cdot \gamma_2, (\gamma_2 \cdot y_1) \# y_2)$ where `$\#$' denotes concatenation of paths. 
Thus there is a short exact sequence
\[\xymatrix@R=0em{1 \ar[r] &\pi_1(\Omega^+_\amp(N),p) \ar[r] & \pi_1(\cM_\amp(N),[p]) \ar[r] & \Gamma^+(N) \ar[r] & 1\\
& y \ar@{|->}[r] &(1,y) && \\
&&(\gamma,y) \ar@{|->}[r] & \gamma. &}\] 

If $\Gamma^+(N)$ acts properly discontinuously then this stacky fundamental group coincides with the fundamental group of the manifold $\Omega^+(N)/\Gamma^+(N)$; more generally, if $\cM(N)$ is an orbifold (as will be the case in this paper) then the stacky fundamental group coincides with the orbifold fundamental group. 

\subsection{K\"ahler moduli space of a complex $K3$}\label{sec:mkahk3}

A complex $K3$ surface is a simply connected compact complex surface $X$ with trivial canonical sheaf $K_X \cong \mathcal{O}_X$. 
General references on $K3$ surfaces include \cite{BHPV, Huybrechts:K3book}. 

The Picard lattice of a complex $K3$ is $\Pic(X) = \NS(X) = H^2(X,\Z)\cap H^{1,1}(X,\C)$, equipped with the cup-product pairing; this is a lattice of rank $0 \leq \rho(X) \leq 20$. If $X$ is algebraic it has signature $(1,\rho(X) - 1)$ (non-algebraic $K3$ surfaces may have Picard rank zero). 
The cohomology $H^*(X,\Z)$ carries a polarized weight-two Hodge structure, whose algebraic part is
\[
\N(X) \coloneqq H^0(X, \Z) \oplus \Pic(X) \oplus H^4(X, \Z), 
\]
and whose polarization is given by the Mukai symmetric form
\[\langle (r_1,D_1,s_1),(r_2,D_2,s_2)\rangle =D_1\cdot D_2-r_1s_2 -r_2s_1.\] 
If $X$ is algebraic, then $\N(X) \cong U \oplus \Pic(X)$ has signature $(2,\rho(X))$. 

\begin{rmk}\label{rmk:Nrelevance}
We briefly explain how the lattice $\N(X)$ is related to the bounded derived category $\Db(X)$. Any object $E$ of $\Db(X)$ has a Mukai vector $v(E) = \mathrm{ch}(E) \sqrt{\mathrm{td} (X)}\in \N(X)$. 
Riemann--Roch takes the form
\[\chi(E,F)\coloneqq\big.\sum_{i} (-1)^i \dim \Hom^i(E,F)=-(v(E),v(F)).\]
As a consequence, any spherical object\footnote{Recall that an object $E$ is called \emph{spherical} if $\Hom^i(E,E)$ is isomorphic to the $i$th cohomology of the 2-sphere: i.e., it has rank one if $i=0,2$ and zero otherwise.} $E$ has Mukai vector $v(E) \in \Delta(\N(X^\circ))$.
\end{rmk}

\begin{defn}
Let $X$ be an algebraic $K3$ surface. Define
\begin{align*}
\Omega^+_0(X) &\coloneqq \Omega^+_0(\N(X)), \\
\Gamma^+(X) & \coloneqq \Gamma^+(\N(X)), \text{ and} \\
\cM_{\text{K\"ah}}(X) & \coloneqq \cM_\amp(\N(X)).
\end{align*}
We call the latter the \emph{K\"ahler moduli space of $X$}.
\end{defn}

\subsection{Complex moduli space of a symplectic $K3$}\label{sec:mcpx}

A \emph{symplectic $K3$} is a symplectic manifold $(X,\omega)$ which is symplectomorphic to a complex $K3$ surface equipped with a K\"ahler form. 
Given a symplectic $K3$, we define the lattice
\[ \bL(X,\omega) \coloneqq [\omega]^\perp \cap H^2(X,\Z),\]
equipped with the intersection pairing. 

\begin{rmk}\label{rmk:Lrelevance}
We briefly explain how the lattice $\bL(X,\omega)$ is related to the Fukaya category. 
Any oriented Lagrangian submanifold of $X$ (and in particular, any object of the Fukaya category) has a homology class $[L] \in \bL(X,\omega)$. 
If $K$ and $L$ are objects of the Fukaya category, we have
\begin{equation}\label{eq:euler_fuk}
 \chi(K,L) \coloneqq \sum_i (-1)^i \dim \Hom^i(K,L) = -[K]\cdot [L],
 \end{equation}
because both sides can be identified as a signed count of intersection points between $K$ and $L$.\footnote{More generally, we have $\chi(K,L) = (-1)^{n(n+1)/2} [K]\cdot [L]$ where $n$ is half the real dimension on $X$.} 
As a consequence, any spherical object (in particular, any Lagrangian sphere) has homology class $[L] \in \Delta(\bL(X,\omega))$.
\end{rmk}

\begin{defn}\label{defn:mcpx}
Suppose that $\bL(X,\omega)$ has signature $(2,t)$. 
Define
\begin{align*}
\Omega^+_0(X,\omega) &\coloneqq \Omega^+_0(\bL(X,\omega)), \\
\Gamma^+(X,\omega) & \coloneqq \Gamma^+(\bL(X,\omega)), \text{ and} \\
\cM_{\text{cpx}}(X,\omega) & \coloneqq \cM_\amp(\bL(X,\omega)).
\end{align*}
We call the latter the \emph{complex moduli space of $(X,\omega)$}.
\end{defn}

Remarks \ref{rmk:Nrelevance} and \ref{rmk:Lrelevance} give some evidence for the following conjecture, which is implicit in \cite{Dolgachev}:

\begin{conj}\label{conj:lattmirr}
If a symplectic $K3$ $(X,\omega)$ is mirror to a complex $K3$ $X^\circ$, then we have an isometry
\[ \bL(X,\omega) \cong \N(X^\circ).\]
As a consequence, when $(X,\omega)$ and $X^\circ$ are mirror there is an isomorphism
\[ \cM_{\text{cpx}}(X,\omega) \cong \cM_{\text{K\"ah}}(X^\circ).\]
\end{conj}

\begin{rmk}
It may seem peculiar that Definition \ref{defn:mcpx} is only given when $\bL(X,\omega)$ has signature $(2,t)$. 
However this condition is natural in light of Conjecture \ref{conj:lattmirr}, which shows that it is expected to hold whenever the mirror $X^\circ$ is algebraic.
A version of the complex moduli space, which does not require the assumption on the signature of $\bL(X,\omega)$, is defined in Section \ref{sec:losemark}. 
(It is explained in Section \ref{subsec:shrink} why $\cM_{\text{cpx}}(X,\omega)$ is superior when it is defined.)
\end{rmk}

\begin{rmk}
We briefly explain the connection with \cite{Dolgachev}. 
If $M$ is a lattice of signature $(1,t)$, Dolgachev defines an ``$M$-polarized $K3$ surface'' to be one equipped with an embedding $M \hookrightarrow \Pic(X)$. 
Given an embedding $M \hookrightarrow L$, he identifies $\Omega^+_0(M^\perp)$ as the ``moduli space of ample marked $M$-polarized $K3$ surfaces'', and $\cM_0(M^\perp)$ as the ``moduli space of ample $M$-polarized $K3$ surfaces''. 
He claims that an $M$-polarized $K3$ surface $X$, equipped with a K\"ahler form $\omega$ satisfying $[\omega] \in M_\R$, should be mirror to an $M^\circ$-polarized $K3$ surface $X^\circ$, where $M^\perp \cong U \oplus M^\circ$. 
If $[\omega] \in M_\R$ is ``irrational'', then $\bL(X,\omega) = M^\perp$; and if $X^\circ$ is very general, then $\Pic(X^\circ) = M^\circ$. 
This gives rise to Conjecture \ref{conj:lattmirr}.
\end{rmk}

\subsection{Picard rank one case}\label{sec:pic1eg}

We will be interested in the case that $(X,\omega)$ is a K\"ahler $K3$ surface mirror to an algebraic $K3$ surface $X^\circ$. 
According to Conjecture \ref{conj:lattmirr}, this means that
\[ \bL(X,\omega) \cong \N(X^\circ) \cong U \oplus \Pic(X^\circ)\]
where $\Pic(X^\circ)$ has signature $(1,\rho(X^\circ) - 1)$ and is furthermore even (since the $K3$ lattice is even). 
The smallest such examples are $\Pic(X^\circ) = \langle 2n \rangle$ for $n \ge 1$. 
Dolgachev has computed the relevant moduli spaces explicitly in this case \cite[Section 7]{Dolgachev}.

The space $\Omega^+(U \oplus \langle 2n\rangle)$ is isomorphic to the upper half plane $\mathfrak{h} \coloneqq \{z \in \C: \im z > 0\}$ via the map
\begin{align*}
\mathfrak{h} & \to \Omega^+(U \oplus \langle 2n\rangle) \\
z & \mapsto [-nz^2:1:z].
\end{align*}
The set $\Delta(U \oplus \langle 2n \rangle)$ consists of all $(a,b,c) \in \Z^3$ satisfying $ab+nc^2+1 = 0$. 
Given such $\delta = (a,b,c)$, the intersection of $\bP(\delta^\perp_\C)$ with $\Omega^+(U \oplus \langle 2n \rangle)$ corresponds to the single point $p_\delta = c/b + i/b\sqrt{n} \in \mathfrak{h}$.
Therefore, the subset $\Omega^+_0(U \oplus \langle 2n \rangle)$ corresponds to the subset
\[ \mathfrak{h}_0 \coloneqq \mathfrak{h} \setminus \left\{\left.\frac{c}{b} + \frac{i}{b\sqrt{n}}\right| c \in \Z, \,b \in \Z_{>0},\, b|nc^2+1\right\}.\]
The group $\Gamma^+(U \oplus \langle 2n\rangle)$ is isomorphic to the Fricke modular group $\Gamma_0^+(n) \subset \mathrm{PSL}(2,\R)$, which is generated by the matrices 
\[
\left( \begin{array}{cc} a&b \\ c&d\end{array}\right) \in \mathrm{SL}(2,\Z) \, \textrm{ with} \  n|c , \quad \textrm{and} \ \left( \begin{array}{cc} 0& -1/\sqrt{n} \\ \sqrt{n}&0\end{array}\right).
\]

\subsection{More moduli spaces}\label{sec:auxil}

Still assuming $N$ to be a lattice of signature $(2,t)$, we define
\[ \EuQ(N) \coloneqq \left\{ \Omega \in N_\C| (\Omega,\Omega) = 0, (\Omega,\overline{\Omega})>0\right\}.\]
It is clearly a $\C^*$-bundle over $\Omega(N)$. 

Now observe that $\EuQ(N)$ consists of classes $\Omega \in N_\C$ satisfying
\[ \left(\mathrm{Re} \,\Omega\right)^2 = \left(\mathrm{Im}\, \Omega\right)^2 > 0, \quad \left(\mathrm{Re} \,\Omega, \mathrm{Im}\, \Omega\right) = 0.\]
Thus it can be identified with the set of conformal bases for positive-definite two-planes in $N_\R$. 
We now define $\EuP(N) \subset N_\C$ to be the open subset of vectors whose real and imaginary parts span a positive-definite two-plane.

\begin{lem}\label{lem:PQdefret}
The subspace $\EuQ(N) \subset \EuP(N)$ is a deformation retract. 
\end{lem}
\begin{proof}
By Gram--Schmidt.
\end{proof}

Furthermore, by taking conformal bases, one sees that $\EuP(N)$ is a $\mathrm{GL}^+(2,\R)$-bundle over $\Omega(N)$ and $\EuQ(N) \subset \EuP(N)$ a $\mathrm{CO}^+(2,\R)$-subbundle, where $\mathrm{CO}^+(2,\R) \cong \R^+ \times \mathrm{SO}(2,\R)$ is the group of orientation-preserving conformal transformations of $\R^2$. 
We denote the component of $\EuP(N)$ lying over $\Omega^\pm(N)$ (respectively, $\Omega^\pm_\amp(N)$) by $\EuP^\pm(N)$ (respectively, $\EuP^\pm_\amp(N)$), and define $\EuQ^\pm_\amp(N) \subset \EuQ^\pm(N)$ similarly.

\begin{defn}
We define
\begin{align*}
\cL_0(N) &\coloneqq [\EuQ_0^+(N)/\Gamma^+(N)],\\
\cL_{\text{K\"ah}}(X^\circ) &\coloneqq \cL_0(\N(X^\circ)) \quad \text{if $X^\circ$ is an algebraic $K3$, and}\\
\cL_{\text{cpx}}(X,\omega) &\coloneqq \cL_0(\bL(X,\omega)) \quad \text{if $\bL(X,\omega)$ has signature $(2,t)$.}
\end{align*}
\end{defn}

\section{Symplectic mapping class groups of $K3$ surfaces} 

In this section we explain how monodromy can be used to construct classes in the symplectic mapping class group of a $K3$ surface equipped with a K\"ahler form.

\subsection{Symplectic monodromy of K\"ahler manifolds}\label{sec:monod}

Let $p:\EuX \to B$ be a family of compact complex manifolds equipped with a locally-constant family $\kappa$ of K\"ahler classes. I.e., $p$ is a proper holomorphic submersion between complex manifolds, and $\kappa_b \in H^2(X_b;\R)$ is locally constant and K\"ahler for each $b \in B$.
Then there exists a smoothly-varying family of K\"ahler forms $\omega_b \in \Omega^{1,1}(X_b;\R)$ representing the K\"ahler classes $\kappa_b$, by \cite[Section 9.3.3]{VoisinI}. 
Moser's theorem says that the family of symplectic manifolds $(X_b,\omega_b)$ is symplectically locally trivial, so there is a monodromy map
\begin{equation}
\label{eq:monod}
\pi_1(B,b) \to G(X_b,\omega_b).
\end{equation}
Because the space of K\"ahler forms representing a given K\"ahler class is convex and therefore contractible, this does not depend on the choice of forms $\omega_b$ (again by Moser).

We require a generalization of this construction: suppose that we have a discrete group $\Gamma$ acting on the family $p:\EuX \to B$, respecting the classes $\kappa_b$ (i.e., $\Gamma$ acts on $\EuX$ and $B$ by biholomorphisms, so that $p$ is equivariant, and $\gamma^* \kappa_{\gamma \cdot b} = \kappa_b$). 
Then we obtain a monodromy map
\begin{equation}
\label{eq:stackmon}
 \pi_1([B/\Gamma],[b]) \to G(X_b,\omega_b),
\end{equation}
which sends a class $(\gamma,y)$ to the composition 
\[(X_b,\omega_b) \xrightarrow{\text{parallel transport along $y$}} (X_{\gamma \cdot b},\omega_{\gamma \cdot b}) \xrightarrow{\gamma^{-1}} (X_b,\gamma^* \omega_{\gamma \cdot b}) \xrightarrow{\text{Moser}} (X_b,\omega_b).\]
The final arrow is defined using Moser's theorem, and the fact that $\gamma^* \omega_{\gamma \cdot b}$ and $\omega_b$ are K\"ahler forms representing the same K\"ahler class. 

By construction, the morphisms \eqref{eq:monod} and \eqref{eq:stackmon} fit into a commutative diagram
\[ \xymatrix{ 1 \ar[r] & \pi_1(B,b) \ar[d]_-{\eqref{eq:monod}} \ar[r] & \pi_1([B/\Gamma],[b]) \ar[d]_-{\eqref{eq:stackmon}} \ar[r] & \Gamma \ar[r] & 1 \\
& G(X_b,\omega_b) \ar@{=}[r] & G(X_b,\omega_b) &&}\]
where the top row is exact.

\subsection{Moduli space of $K3$ surfaces}\label{sec:k3mods}

A \emph{marked $K3$ surface} is a pair $(X,\phi)$ with $X$ a $K3$ surface and $\phi: H^2(X,\Z) \rightarrow L$ a lattice isometry. 
A marked $K3$ surface $(X,\phi)$ has a period point
\begin{equation} \label{eqn:period}
\phi_\C(H^{2,0}(X)) \in \Omega(L).
\end{equation}
There is a moduli space of marked $K3$ surfaces, and the period mapping \eqref{eqn:period} defines a local isomorphism between this moduli space and the 20-dimensional complex manifold $\Omega(L)$. 
The period mapping is furthermore surjective, however it is not injective, and the moduli space is not Hausdorff (see \cite[Section VIII.12]{BHPV}). 

We define a \emph{K\"ahler $K3$ surface} to be a $K3$ surface $X$ equipped with a K\"ahler class $\kappa \in H^{1,1}(X;\R)$; thus a \emph{marked K\"ahler $K3$ surface} is a triple $(X,\phi,\kappa)$ where $(X,\phi)$ is a marked $K3$ surface and $\kappa \in H^{1,1}(X;\R)$ is a K\"ahler class.
We define
\[ \mathcal{K}(L) := \left\{ (\Omega,\kappa) \in \Omega(L) \times L_\R: \kappa^2>0, (\Omega,\kappa)=0\right\}.\]
A marked K\"ahler $K3$ surface $(X,\phi,\kappa)$ has a period point
\begin{equation}
\label{eq:kperiod}
\left(\phi_\C(H^{2,0}(X)),\phi_\R(\kappa)\right) \in \mathcal{K}(L).
\end{equation}
There is a moduli space of marked K\"ahler $K3$ surfaces, and the period mapping \eqref{eq:kperiod} defines an isomorphism with the subspace 
\begin{equation}
\label{eqn:kamp}
 \mathcal{K}_\amp(L) \coloneqq \mathcal{K}(L) \setminus \bigcup_{\delta \in \Delta(L)} \bP(\delta^\perp_\C) \times \delta^\perp_\R
 \end{equation}
(see \cite[Theorems VIII.12.3 and VIII.14.1]{BHPV}).

\begin{defn}
Given $\kappa \in L_\R$ satisfying $\kappa^2>0$, we define
\begin{align*}
\Omega(\kappa) &\coloneqq \left\{ \Omega \in \Omega(L): (\Omega,\kappa)=0\right\} \subset \Omega(L)\\
\Omega_\amp(\kappa) & \coloneqq \Omega(\kappa) \setminus \bigcup_{\delta \in \Delta(L) \cap \kappa^\perp} \bP(\delta^\perp_\C).
\end{align*}
We observe that $\kappa^\perp$ has signature $(2,19)$, and therefore $\Omega(\kappa) \subset \bP((\kappa^\perp)_\C)$ has two connected components $\Omega^\pm(\kappa)$, distinguished by the orientation of the positive-definite two-plane in $\kappa^\perp$ spanned by the real and imaginary parts of $\Omega$ and interchanged by complex conjugation. 
We similarly denote the connected components of $\Omega_\amp(\kappa)$ by $\Omega^\pm_\amp(\kappa)$.
\end{defn}

Now the moduli space of marked K\"ahler $K3$ surfaces is \emph{fine}, meaning it carries a universal family $\EuX \to \mathcal{K}_\amp(L)$. 
It is clear from the definitions that $ \Omega^+_\amp(\kappa) \times \{\kappa\} \subset \mathcal{K}_\amp(L)$, so we may restrict the universal family to this submanifold to obtain a family of $K3$ surfaces $\EuX(\kappa) \to \Omega^+_\amp(\kappa)$ equipped with a locally-constant family of K\"ahler classes $\phi_\R^{-1}(\kappa)$. 
By the previous section we obtain a map
\[ \pi_1\left(\Omega^+_\amp(\kappa)\right) \to G(X,\omega)\]
where $\omega$ is a K\"ahler form with class $\kappa$. 
In fact, because the cohomology of the family is globally trivialized by the marking $\phi$, the image of this map lies in $I(X,\omega) \subset G(X,\omega)$.

\begin{rmk}\label{rmk:dehntwist}
Let us consider a neighbourhood of a point on one of the hypersurfaces $\bP(\delta^\perp_\C) \cap \Omega^+(\kappa)$ that get removed to obtain $\Omega^+_0(\kappa)$, which does not intersect any of the other hypersurfaces. 
Inside this neighbourhood, the family $\EuX(\kappa)$ can be extended across the hypersurface, at the expense that it develops a node along the hypersurface. 
The resulting vanishing cycle is a Lagrangian sphere $L$ with homology class satisfying $\phi([L]) = \pm\delta$, and the symplectic monodromy around the hypersurface is the squared Dehn twist $\tau_L^2$. 
In particular, every class in $\Delta(\N(X,\omega))$ is the homology class of a vanishing cycle.
\end{rmk}

\subsection{Getting rid of the marking}\label{sec:losemark}

The fibres of the family $\EuX(\kappa) \to \Omega^+_0(\kappa)$ are marked K\"ahler $K3$ surfaces. 
We now determine which fibres are isomorphic as unmarked K\"ahler $K3$ surfaces. 
We do this by applying the Torelli theorem for $K3$ surfaces \cite[Chapter VIII]{BHPV}, which states that two K\"ahler $K3$ surfaces $(X_0,\kappa_0)$ and $(X_1,\kappa_1)$ are isomorphic if and only if there is a Hodge isometry $\psi: H^2(X_0,\Z) \xrightarrow{\sim} H^2(X_1,\Z)$ such that $\psi_\R(\kappa_0) = \kappa_1$ (and the isomorphism, if it exists, is unique).

\begin{defn}
Let $\kappa \in L_\R$ satisfy $\kappa^2>0$. Define
\[ \Gamma(\kappa) \coloneqq \left\{\sigma \in \Aut(L): \sigma_\R(\kappa) = \kappa\right\}.\]
It is clear that $\Gamma(\kappa)$ acts on $\Omega(\kappa)$; we define $\Gamma^+(\kappa) \subset \Gamma(\kappa)$ to be the subgroup preserving $\Omega^+(\kappa)$. 
\end{defn}

\begin{lem}\label{lem:applytor}
The fibres of the family $\EuX(\kappa) \to \Omega^+_0(\kappa)$ over points $b_0,b_1 \in \Omega^+_0(\kappa)$ are isomorphic, as unmarked K\"ahler $K3$ surfaces, if and only if $\gamma \cdot b_0 = b_1$ for some $\gamma \in \Gamma^+(\kappa)$. When the isomorphism exists, it is unique.
\end{lem}
\begin{proof}
Follows from the Torelli theorem: the required Hodge isometry is $\psi = \phi_1^{-1} \circ \sigma \circ \phi_0$.
\end{proof}

\begin{cor}
The action of $\Gamma^+(\kappa)$ on $\Omega^+_\amp(\kappa)$ extends uniquely to an action on the family $\EuX(\kappa) \to \Omega^+_\amp(\kappa)$, respecting K\"ahler classes.
\end{cor}
\begin{proof}
Given $\sigma \in \Gamma^+(\kappa)$, we consider the family of marked K\"ahler $K3$ surfaces $\EuX(\kappa)_\sigma \to \Omega^+_\amp(\kappa)$ obtained from $\EuX(\kappa) \to \Omega^+_\amp(\kappa)$ by post-composing the marking with $\sigma^{-1}$. 
Because $\EuX(\kappa) \to \Omega^+_\amp(\kappa)$ is a universal family (see \cite[Chapter VIII]{BHPV}), $\EuX(\kappa)_\sigma$ is classified by a unique map $\Omega^+_0(\kappa) \to \Omega^+_0(\kappa)$. 
By Lemma \ref{lem:applytor}, this map coincides with the action of $\sigma$ on $\Omega^+_0(\kappa)$; and the resulting isomorphism $\EuX(\kappa)_\sigma \cong \sigma^* \EuX(\kappa)$ provides the desired lift of the action of $\sigma$ to $\EuX(\kappa)$.
\end{proof}

Now let $(X,\phi,\kappa)$ be a marked K\"ahler $K3$ surface, and $\omega$ a K\"ahler form representing $\phi_\R^{-1}(\kappa)$. 
It is clear that the image of the map
\begin{equation}\label{eq:sympacthom}
 G(X,\omega) \to \Aut (H^2(X,\Z)) \overset{\phi}{\simeq} \Aut (L)
 \end{equation}
lies in $\Gamma(\kappa)$. 
We have the following result of Donaldson:

\begin{prop}[\cite{Donaldson:PolyInv}, Proposition 6.2]\label{prop:Donplus}
The image of \eqref{eq:sympacthom} lies in $\Gamma^+(\kappa)$.
\end{prop}

It is clear from the definition that the action of $\Gamma^+(\kappa)$ respects the K\"ahler classes of the family $\EuX(\kappa) \to \Omega^+_\amp(\kappa)$, so the construction of Section \ref{sec:monod} gives a commutative diagram
\begin{equation}
\label{eqn:sesmonodromy1}
\xymatrix{ 1 \ar[r] & \pi_1(\Omega^+_\amp(\kappa)) \ar[r] \ar[d] & \pi_1([\Omega^+_\amp(\kappa)/\Gamma^+(\kappa)]) \ar[r] \ar[d] & \Gamma^+(\kappa) \ar[r] \ar@{=}[d] & 1 \\
1 \ar[r] & I(X,\omega) \ar[r] & G(X,\omega) \ar[r] & \Gamma^+(\kappa) \ar[r] & 1.} 
\end{equation}
Proposition \ref{prop:Donplus} allows us to write the bottom $\Gamma^+(\kappa)$ instead of $\Aut(H^2(X,\Z))$; the bottom `$\to 1$' can then be added because the map $G(X,\omega) \to \Gamma^+(\kappa)$ is surjective by commutativity of the diagram (cf. \cite[Proposition 6.1]{Donaldson:PolyInv}).

\subsection{Shrinking the family}\label{subsec:shrink}

The family $\EuX(\kappa) \to \Omega^+_\amp(\kappa)$ is the universal family of marked K\"ahler $K3$ surfaces with K\"ahler class $\kappa$ (more precisely, the universal family is the disjoint union of this together with its complex conjugate), and the group $\Gamma^+(\kappa)$ is as large as possible by Lemma \ref{lem:applytor}, so the monodromy maps appearing in \eqref{eqn:sesmonodromy1} are ``the best we can do using the construction of Section \ref{sec:monod}''. 
However it turns out that in many interesting situations, including those we study in this paper, the family $\EuX(\kappa) \to \Omega^+_0(\kappa)$ is ``bigger than it needs to be'': we can construct a submanifold $\Omega^+_0(N_\kappa) \subset \Omega^+_0(\kappa)$ which is a deformation retract. 
Thus for the purposes of studying the symplectic mapping class group, we might just as well consider the monodromy homomorphism arising from the restriction of the family $\EuX(\kappa) \to \Omega^+_0(\kappa)$ to this submanifold. 
The benefit is twofold: firstly, it is easier to understand $\Omega^+_0(N_\kappa)$ (in the main examples considered in this paper, it is one-dimensional whereas $\Omega^+_0(\kappa)$ is nineteen-dimensional), and hence compute its fundamental group; secondly, it is directly comparable with the moduli space of stability conditions on the mirror.

\begin{lem}\label{lem:defret}
Suppose that $\kappa \in L_\R$ satisfies $\kappa^2>0$, and $N_\kappa \coloneqq \kappa^\perp \cap L$ is a sublattice of signature $(2,t)$. 
Then $\Omega^+_\amp(N_\kappa) \subset \Omega^+_\amp(\kappa)$ is a deformation retract.
\end{lem}
\begin{proof}
We first check that $\Omega^+_\amp(N_\kappa)$ does lie inside $\Omega^+_\amp(\kappa)$: this follows from the fact that $\Delta(N_\kappa) = \Delta(L) \cap \kappa^\perp$. 
Now let $\EuP(\kappa) \subset L_\C$ be the set of vectors which are orthogonal to $\kappa$ and whose real and imaginary parts span a positive-definite two-plane. 
It has two connected components, and we define $\EuP^+(\kappa)$ to be the one containing $\EuP^+(N_\kappa)$. 
We define 
\[ \EuP^+_\amp(\kappa) \coloneqq \EuP^+(\kappa) \setminus \bigcup_{\delta \in \Delta(N_\kappa)} \delta^\perp_\C,\]
and observe that it is a $\mathrm{GL}^+(2,\R)$-bundle over $\Omega^+_\amp(\kappa)$ as in Section \ref{sec:auxil}.

Now observe that we have an orthogonal direct sum decomposition
\[ \kappa^\perp = (N_\kappa)_\R \oplus \left((N_\kappa^\perp)_\R \cap \kappa^\perp\right),\] 
and the second summand is negative definite because $L$ has signature $(3,19)$, $N_\kappa$ has signature $(2,t)$, and $\kappa^2>0$. 
This gives a direct sum decomposition of $\kappa^\perp_\C$, which we use to define a linear retraction onto $(N_\kappa)_\C$.  Negative-definiteness of the second summand implies that this retraction respects $\EuP^+(\kappa)$, and the fact that the second summand is orthogonal to $\Delta(N_\kappa)$ implies that it respects $\EuP^+_\amp(\kappa)$; hence it retracts the latter onto $\EuP^+_\amp(N_\kappa)$. 
The retraction is clearly $\mathrm{GL}^+(2,\R)$-equivariant, so descends to the required retraction of $\Omega^+_\amp(\kappa)$ onto $\Omega^+_\amp(N_\kappa)$.  
\end{proof}

We define $\EuX(N_\kappa,\kappa) \to \Omega^+_0(N_\kappa)$ to be the restriction of the family $\EuX(\kappa) \to \Omega^+_0(\kappa)$ to $\Omega^+_0(N_\kappa)$.
We remark that this family of marked $K3$ surfaces does depend on $\kappa$ (cf. \cite[Theorem 3.1]{Dolgachev}).

\begin{lem}\label{lem:gammasame}
If $N_\kappa \subset L$ is a sublattice of signature $(2,t)$, then $\Gamma^+(N_\kappa) \cong \Gamma^+(\kappa)$.
\end{lem}
\begin{proof}
The isomorphism \eqref{eqn:gammanperp} identifies $\Gamma(N_\kappa)$ with
\[\left\{\sigma \in \Aut(L): \sigma|_{N_\kappa^\perp} = \id\right\},\]
and this is equal to $\Gamma(\kappa)$. 
Indeed, if $\sigma$ fixes $N_\kappa^\perp$ then it fixes $\kappa \in (N_\kappa^\perp)_\R$; conversely, if $\sigma \in \Aut(L)$ fixes $\kappa$, then $\ker(\sigma - \id)$ is a primitive subgroup containing $\kappa$ in its $\R$-span, and it is straightforward to show that $N_\kappa^\perp$ is the smallest such subgroup. 
It is clear that the subgroups $\Gamma^+(N_\kappa)$ and $\Gamma^+(\kappa)$ correspond under this identification.
\end{proof}

Putting Lemmas \ref{lem:defret} and \ref{lem:gammasame} together, we have shown:

\begin{prop}\label{prop:shrunkmon}
Let $X$ be a $K3$ surface and $\omega$ a K\"ahler form, such that $\bL(X,\omega)$ is a sublattice of signature $(2,t)$.  
Then we have a morphism of short exact sequences
\begin{equation}
\label{eqn:sesmonodromy2}
\xymatrix{ 1 \ar[r] & \pi_1(\Omega^+_\amp(X,\omega)) \ar[r] \ar[d] & \pi_1(\cM_{\text{cpx}}(X,\omega)) \ar[r] \ar[d] & \Gamma^+(X,\omega) \ar[r] \ar@{=}[d] & 1 \\
1 \ar[r] & I(X,\omega) \ar[r] & G(X,\omega) \ar[r] & \Gamma^+(X,\omega) \ar[r]& 1.} 
\end{equation}
This diagram is isomorphic to \eqref{eqn:sesmonodromy1}.
\end{prop}

\subsection{Examples}\label{sec:eg_fam}

We consider K\"ahler $K3$ surfaces $(X,\omega)$ which are mirror to an algebraic $K3$ surface $X^\circ$ of the smallest possible Picard rank $\rho(X^\circ) = 1$. 
According to Section \ref{sec:pic1eg}, this means we have
\[ \bL(X,\omega) \cong \N(X^\circ) \cong U \oplus \langle 2n \rangle\]
for some positive integer $n$. 
We will focus on the cases $n=1,2$. 
The mirrors $X^\circ$ arise geometrically as follows:

\begin{enumerate}
\item A very general quartic hypersurface $X^\circ \subset \bP^3$ has $\Pic(X^\circ) \cong \langle 4\rangle$;  \item  A very general sextic hypersurface $X^\circ \subset \bP(1,1,1,3)$ (also called a ``double plane'') has $\Pic(X^\circ) \cong \langle 2\rangle$.\end{enumerate}
Note that in the second case, $\bP(1,1,1,3)$ has an isolated singularity, and the generic hypersurface is disjoint from that point and hence smooth. 

We will now give explicit descriptions of $\cM_0(U \oplus \langle 2n \rangle) \cong [\mathfrak{h}_0/\Gamma^+_0(n)]$ in the cases $n=1,2$. 
We will also give explicit descriptions of the universal families 
\[ [\EuX(U \oplus \langle 2n \rangle,\kappa)/\Gamma^+(U \oplus \langle 2n \rangle)] \to \cM_0(U \oplus \langle 2n \rangle)\] 
for certain $\kappa$, in these cases. 
The families will be constructed as families of hypersurfaces $X\subset Y$, where $Y$ is a certain toric variety. 
The intersections of $X$ with the toric boundary divisors (which will always be transverse) span a sublattice $\Pic_{tor}(X) \subset \Pic(X)$. 
We say that a K\"ahler form $\omega$ on $X$ is ``ambient'' if it is the restriction of a K\"ahler form on $Y$, and ``ambient-irrational'' if it is ambient and furthermore $[\omega] \in \Pic_{tor}(X)_\R$ is not contained in any proper rational subspace.
In the latter case, we have
\[ \bL(X,\omega) = \Pic_{tor}(X)^\perp.\]

\begin{example}[The case $n=2$] \label{Ex:mirror_quartic} 
We consider the ``Dwork family''
\begin{equation} \label{eqn:Dwork}
Q_{\lambda} = \left \{x_0^4 +x_1^4+x_2^4+x_3^4+ 4\lambda x_0x_1x_2x_3 = 0\right \} \subset \bP^3
\end{equation}
of quartics, parametrized by $\lambda \in \bA^1$. These are all invariant under the group $\Pi = \Z/4\times\Z/4$ of diagonal projective transformations $[i^{a_0},i^{a_1},i^{a_2},i^{a_3}]$ of $\bP^3$ with $\sum a_j = 0$.  If $\lambda^4\neq 1$ then the quotient $Q_{\lambda}/\Pi$ has six $A_3$-singularities; the fibres over $\lambda^4=1$ have an additional nodal singularity. There is a  toric resolution of the ambient toric variety $\bP^3/\Pi$, yielding a simultaneous resolution of the Dwork family, which defines a family of $K3$'s over $\bA^1 \backslash \{\lambda^4=1\}$. This is the Greene--Plesser mirror family to the family of smooth quartics in $\bP^3$.

A member of this family is called a ``mirror quartic''. 
Any mirror quartic $X$ satisfies $\Pic_{tor}(X)^\perp \cong U \oplus \langle 4 \rangle$ (see \cite[Theorem 8.2]{Dolgachev}, and also \cite{Rohsiepe}). 
Therefore, if $\omega$ is an ambient-irrational K\"ahler form, we have 
\begin{align*}
\bL(X,\omega) &\cong U \oplus \langle 4 \rangle, \text{ and in particular,}\\
\cM_{\text{cpx}}(X,\omega) & \cong \cM_0(U \oplus \langle 4 \rangle).
\end{align*}
There is an action of $\Z/4$ on the family by multiplying one coordinate by $i$, which covers the action $\lambda \mapsto i\lambda$ on $(\bA^1 \backslash \{\lambda^4 = 1\})$.  
Dolgachev proves that the quotient $[\bA^1/(\Z/4)]$ is in fact isomorphic to $\cM(U \oplus \langle 4 \rangle)$, and the complement of $\{\lambda^4=1\}$ corresponds to $\cM_0(U \oplus \langle 4 \rangle)$. 
He furthermore identifies the quotient of the family by this $\Z/4$-action, with a family $[\EuX(U \oplus \langle 4 \rangle,\kappa)/\Gamma^+(U \oplus \langle 4 \rangle)]$ for appropriate $\kappa$.

We can compactify $[\bA^1/(\Z/4)]$ to an orbifold $\bP^1$, by adding a point at $\lambda^4=\infty$. 
The orbifold has 3 special points: the ``cusp'' $\lambda^4 = \infty$ which must be removed to get $\cM(U \oplus \langle 4 \rangle)$; the ``nodal point'' $\lambda^4 = 1$ which must further be removed to get $\cM_{\amp}(U \oplus \langle 4 \rangle)$; and the order-4 orbifold point $\lambda^4 = 0$. 
\end{example}

\begin{example}[The case $n=1$] \label{Ex:mirror_double_plane} 
The one-dimensional family of hypersurfaces
\begin{equation} \label{eqn:non-Dwork}
P_{\lambda} = \left\{ x_0^6 + x_1^6 + x_2^6 + x_3^2 + \lambda x_0x_1x_2x_3 \right\} \subset \bP(1,1,1,3)
\end{equation}
are all invariant under a group $\Pi' = \Z/6 \times \Z/2$ of diagonal projective transformations. The simultaneous crepant resolution of the quotient defines the Greene--Plesser mirror family to the family of double planes, and its fibres are called ``mirror double planes''. 
Any mirror double plane $X$ satisfies $\Pic_{tor}^\perp \cong U \oplus \langle 2 \rangle$ (see \cite[Example 8.3]{Dolgachev} and \cite{Rohsiepe}). 
In particular, if $\omega$ is an ambient-irrational K\"ahler form then we have $\cM_{\text{cpx}}(X,\omega) \cong \cM_0(U \oplus \langle 2 \rangle)$.

Similarly to the previous example, $\cM(U \oplus \langle 2 \rangle)$ compactifies to an orbifold $\bP^1$ with 3 special points: the cusp, a nodal point, and an order-3 orbifold point.
\end{example} 

\subsection{Consequences}

We now record some consequences of the explicit descriptions of $\cM_\amp(U \oplus \langle 2n \rangle)$ that were given in Examples \ref{Ex:mirror_quartic} and \ref{Ex:mirror_double_plane}. 
These can of course be proven more directly, by the techniques that Dolgachev uses to arrive at these explicit descriptions.

\begin{lem}
\label{lem:oneorbit}
Suppose $N = U \oplus \langle 2 \rangle$ or $U \oplus \langle 4 \rangle$. 
Then $\Gamma^+(N)$ acts transitively on $\Delta(N)/\pm \id$.
\end{lem}
\begin{proof}
There is a bijective correspondence between classes $[\delta] \in \Delta(N)/\pm \id$ and nodal points $p_{\delta} \in \Omega^+(N)$ that must be removed to obtain $\Omega^+_0(N)$: precisely, $\bP(\delta^\perp_\C) \cap \Omega^+(N)= \{p_{\delta}\}$. 
Therefore there is a bijective correspondence between orbits of the action of $\Gamma^+(N)$ on $\Delta(N)/\pm \id$ and nodal points in $\cM(N)$. 
We have seen in Examples \ref{Ex:mirror_quartic} and \ref{Ex:mirror_double_plane} that $\cM(N)$ contains a unique nodal point when $N = U \oplus \langle 2 \rangle$ or $U \oplus \langle 4 \rangle$.
\end{proof}

\begin{lem}\label{lem:free_prod}
Suppose $N = U \oplus \langle 2 \rangle$ (respectively, $N =U \oplus \langle 4\rangle$). 
Then we have identifications 
\[ \xymatrix{ \pi_1(\cM_0(N)) \ar@{<->}[d]^{\rotatebox[origin=c]{90}{$\sim$}} \ar[r]^-{\eqref{eqn:sesmonodromy2}} & \Gamma^+(N) \ar@{<->}[d]^{\rotatebox[origin=c]{90}{$\sim$}} \\
\Z \ast \Z/p \ar[r] & \Z/2 \ast \Z/p,}\]
where $p=3$ (respectively, $p=4$) and the homomorphism on the bottom row is the obvious one.
\end{lem}

\begin{proof}
The identification $\pi_1(\cM_{\amp}) \cong \Z \ast \Z/p$ is immediate from the description of $\cM_{\amp}$ given in Examples \ref{Ex:mirror_quartic} and \ref{Ex:mirror_double_plane}: the factor $\Z$ corresponds to loops around the nodal point, and the factor $\Z/p$ to loops around the orbifold point. 
By the short exact sequence \eqref{eqn:sesmonodromy2}, the kernel of the map $\pi_1(\cM_{\amp}) \to \Gamma^+$ is equal to the image of $\pi_1(\Omega^+_\amp)$. 
We recall that $\Omega^+_{\amp}$ is the complement of the infinite set of points $p_{\delta} \in \mathfrak{h}$ corresponding to $[\delta] \in \Delta(N)/\pm \id$, so its fundamental group is generated by loops around these points. 
The covering group acts transitively on these points by Lemma \ref{lem:oneorbit}, so the image of the fundamental group is generated by elements conjugate to the image of a loop around a single point $p_{\delta}$. 
A loop around a single point $p_{\delta} \in \Omega^+$ maps to a loop going twice around the nodal point in $\cM$, which can be chosen to correspond to the element $2 \in \Z$. 
Therefore the kernel of the map $\Z \ast \Z/p \to \Gamma^+$ is generated by the elements conjugate to $2 \in \Z$, which means the map can be identified with the projection $\Z \ast \Z/p \twoheadrightarrow \Z/2 \ast \Z/p$ as required.
\end{proof}

\begin{cor} \label{cor:not_generated_by_dehn_twists}
If $X$ is a mirror quartic or mirror double plane and $\omega$ is ambient-irrational, then $G(X,\omega)$ is not generated by Dehn twists, although $X$ contains a Lagrangian sphere.
\end{cor}

\begin{proof}
Let $N \coloneqq \bL(X,\omega)$. If $X$ is a mirror quartic (respectively, a mirror double plane), we saw in Example \ref{Ex:mirror_quartic} (respectively, Example \ref{Ex:mirror_double_plane}) that $N = U \oplus \langle 4 \rangle$ (respectively, $N = U \oplus \langle 2 \rangle$).
Applying Proposition \ref{prop:shrunkmon} and Lemma \ref{lem:free_prod}, we have a surjective homomorphism 
\[G(X,\omega) \twoheadrightarrow \Gamma^+(N) \cong \Z/2 \ast \Z/p.\]
The Dehn twist in a given vanishing cycle can be arranged to map to $1 \in \Z/2$ (cf. Remark \ref{rmk:dehntwist}). 
We know that the action of a Dehn twist in $L$ on homology is given by the Picard--Lefschetz reflection in the corresponding homology class up to sign, $[L] \in \Delta(N)/\pm \id$. 
It follows by Lemma \ref{lem:oneorbit} that all Dehn twists map to elements conjugate to $1 \in \Z/2$. 
Passing to abelianizations, it follows that the image of a Dehn twist under the map
\[
G(X,\omega) \twoheadrightarrow (\Z/2 \ast \Z/p)^{ab} = \Z/2 \oplus \Z/p
\]
is $(1,0)$, so Dehn twists can not generate $G(X,\omega)$. 
\end{proof}

\section{Derived autoequivalences of complex $K3$s}\label{sec:derautk3}

Let $X$ be a complex algebraic $K3$ surface, and $\Db(X)$ its bounded derived category. 
In this section we briefly survey what is known about the group $\Auteq \Db(X)$ of triangulated, $\C$-linear autoequivalences of $\Db(X)$, following \cite{BB} closely. There are no original results in this section.

\subsection{Action of autoequivalences on cohomology}

Let $K(\Db(X))$ denote the Grothendieck group of $\Db(X)$. It admits the \emph{Euler form} $(E,F) \mapsto \chi(E,F)$, whose left- and right-kernels coincide by Serre duality:
\[ \chi(E,F)=0 \quad \forall F \quad \iff \quad \chi(F,E)=0 \quad \forall F.\]
The numerical Grothendieck group $K_{num}(\Db(X))$ is the quotient of the Grothendieck group by this kernel. The Euler pairing descends to it, by construction. The association $[E] \mapsto v(E)$ defines an isometry
\begin{equation}
\label{eq:knummuk}
 K_{num}(\Db(X)) \simeq \N(X)^-,
 \end{equation}
cf. Remark \ref{rmk:Nrelevance} (the fact that this is an isomorphism is specific to $K3$ surfaces: see \cite[Section 16.2.4]{Huybrechts:K3book}). 
Triangulated autoequivalences clearly act on the numerical Grothendieck group by isometries, so we obtain a homomorphism
\begin{equation}\label{eq:derautKact} \varpi_K\colon \Auteq \Db(X) \longrightarrow \aut \N(X).
\end{equation}

Recall from Section \ref{sec:mkahk3} that $H^*(X,\Z)$ carries a polarized weight-two Hodge structure, of which $\N(X)$ is the algebraic part.  
The action \eqref{eq:derautKact} of derived autoequivalences on the algebraic part extends to an action on the whole polarized Hodge structure: i.e., there is a group homomorphism
\begin{equation}
\label{eq:derautact}
\varpi\colon \Auteq \Db(X)\longrightarrow \aut H^*(X,\Z)
\end{equation}
to the group of Hodge isometries, such that $\varpi_K = \varpi|_{\N(X)}$.
Indeed, by a theorem of Orlov \cite[Theorem 2.19]{Orl}, any autoequivalence is of Fourier--Mukai type, meaning it has the form $\pi_{2*}(E \otimes \pi_1^*(-))$ for some $E \in \Db(X \times X)$ called the Fourier--Mukai kernel.  
The Mukai vector $v(E) \in H^*(X \times X,\Z)$ of the kernel induces a correspondence, whose action on cohomology preserves the Hodge filtration, the integral structure and the Mukai pairing. 
This defines the desired map $\varpi$. 

The image of $\varpi$ is the index-2 subgroup $\aut^+ H^*(X,\Z)$ consisting of Hodge isometries which preserve orientations of positive-definite 4-planes, by work of Huybrechts, Macr\`i and Stellari: 

\begin{thm}[\cite{HubMacStel}, Theorem 2]\label{thm:cohaut}
Let $\Auteq^0 \Db(X)$ denote the kernel of $\varpi$. 
 Then there is a short exact sequence
\[ 1 \to \Auteq^0 \Db(X) \to \Auteq \Db(X) \xrightarrow{\varpi} \aut^+ H^*(X,\Z) \to 1.\]
\end{thm} 

We now define $\Aut^+_{CY} H^*(X,\Z) \subset \aut^+ H^*(X,\Z)$ to be the subgroup of Hodge isometries acting by the identity on $H^{2,0}(X,\C)$. 
An autoequivalence is called \emph{Calabi--Yau} if its image under $\varpi$ lies in this subgroup, and the subgroup of Calabi--Yau autoequivalences is denoted by $\Auteq_{CY}\Db(X) \subset \Auteq \Db(X)$. 

\begin{cor}\label{cor:cohcyaut}
There is a short exact sequence
\[ 1 \to \Auteq^0 \Db(X) \to \Auteq_{CY} \Db(X) \xrightarrow{\varpi_K} \Gamma^+(X) \to 1.\]
\end{cor} 
\begin{proof}
Note that $H^{2,0}(X,\C) \subset \N(X)^\perp_\C$ is not contained in any proper rational subspace, so any $\phi \in \aut_{CY}H^*(X,\Z)$ acts by the identity on $\N(X)^\perp$. 
Thus we have an identification $\aut_{CY} H^*(X,\Z) \simeq \Gamma(\N(X))$ by Lemma \ref{lem:nikulin}. 
An element of $\aut_{CY} H^*(X,\Z)$ acts by the identity on the positive-definite 2-planes in $(\N(X)^\perp)_\R$, so it preserves orientations of positive-definite 4-planes if and only if the corresponding element of $\Gamma(\N(X))$ preserves orientations of positive-definite 2-planes.  
Therefore the image of $\varpi_K|_{\Auteq_{CY} \Db(X)}$ is identified with $\Gamma^+(\N(X))$, by Theorem \ref{thm:cohaut}.
\end{proof}

\subsection{Stability conditions on $K3$ surfaces}

Fix a triangulated category $\Db$ linear over a field $\fk$. 
Suppose $\Db$ is proper, i.e. that $\oplus_{i \in \Z} \Hom_\Db(E,F[i])$ is finite-dimensional for all objects $E$, $F$. 
We let  $\Stab(\Db)$ denote the set of stability conditions on $\Db$
which  are \emph{numerical}, \emph{full} and \emph{locally-finite} in the terminology of \cite{Bridgeland, Bridgeland-K3}.     Recall that a numerical stability condition $\sigma = (Z,\P)$ on $\Db$ comprises a group homomorphism $Z: K_{num}(\Db) \to \C$, called the central charge, and a collection of full additive subcategories $\P(\phi) \subset \Db$ for $\phi \in \R$ which satisfy various axioms, see \cite{Bridgeland}. The main result of \emph{op. cit.} asserts that the space $\Stab(\Db)$ has the structure of a complex manifold, such that the forgetful map
\begin{equation}
\label{eqn:stablocaliso}
\pi\colon \Stab(\Db)\longrightarrow \Hom_{\Z}(K_{num}(\Db),\C),
\end{equation}
taking a stability condition to its central charge,
is a local isomorphism.\footnote{More precisely, \cite{Bridgeland} shows that for each connected component of the space of \emph{numerical}, \emph{locally-finite} stability conditions, the forgetful map defines a local isomorphism to a linear subspace $V$ of $\Hom_{\Z}(K_{num}(\Db),\C)$, and \cite{Bridgeland-K3} defines such a stability condition to be \emph{full} if the subspace $V$ is all of $\Hom_{\Z}(K_{num}(\Db),\C)$.} The group of triangulated autoequivalences $\Auteq (\Db)$ acts on $\Stab(\Db)$; an element $\Phi \in \aut(\Db)$ acts by
\[\Phi\colon (Z,\P)\mapsto (Z',\P'), \quad Z'(E)=Z(\Phi^{-1}(E)), \quad \P'=\Phi(\P).\] 

Suppose now that $X$ is a complex algebraic K3 surface, and $\Db(X)$ its bounded derived category. 
We denote $\Stab(X) := \Stab(\Db(X))$.  
We can identify $\Hom_\Z(K_{num}(\Db(X)),\C) \cong \N(X) \otimes \C$ via \eqref{eq:knummuk} together with the Mukai pairing, so \eqref{eqn:stablocaliso} becomes a local isomorphism 
\begin{equation}
\label{eqn:pistabN}
\pi: \Stab(X) \to \N(X) \otimes \C.
\end{equation}
This map is $\Auteq \Db(X)$-equivariant, where the action on $\N(X) \otimes \C$ is via $\varpi_K$.

Recall that $\N(X)$ has signature $(2, \rho(X))$.  Define $\EuP_0(X) \coloneqq \EuP_0(\N(X))$ and let $\EuP^+_0(X)$ be the component containing vectors of the
form $(1,i\omega, -\half\omega^2)$ for an ample class $\omega\in \Pic(X)\otimes \R$. 
Let $\Stab^\dag(X)\subset \Stab(X)$ be the connected component containing the (non-empty) set of ``geometric'' stability conditions,
for which all skyscraper sheaves  $\mathcal{O}_x$ are stable of the same phase.

\begin{thm}[\cite{Bridgeland-K3}, Theorem 1.1]\label{thm:Bridgeland_cov}
The map \eqref{eqn:pistabN} induces a normal covering map
\[ \pi: \Stab^\dag(X) \to \EuP_0^+(X),\]
whose group of deck transformations is the subgroup of $\Auteq^0 \Db(X)$ which preserves $\Stab^\dag(X)$. 
\end{thm}

\subsection{Bridgeland's conjecture, and the Bayer--Bridgeland theorem}

Bridgeland has conjectured:

\begin{conj}[\cite{Bridgeland-K3}, Conjecture 1.2]\label{conj:Bridgeland}
The component $\Stab^\dag(X)$ is simply-connected, and preserved by the group $\Auteq \Db(X)$. 
\end{conj}

Bayer and Bridgeland have proved the conjecture in the Picard rank one case:

\begin{thm}[\cite{BB}, Theorem 1.3] \label{thm:bb}
Conjecture \ref{conj:Bridgeland} holds when $\rho(X) = 1$.
\end{thm}

We now state some Corollaries of Theorems \ref{thm:Bridgeland_cov} and \ref{thm:bb}. 
The hypothesis $\rho(X)=1$ can be removed from each if Bridgeland's conjecture holds.

\begin{cor} \label{cor:BBaut}
If $\rho(X) = 1$, then there is an isomorphism of short exact sequences
\[
\xymatrix{ 1 \ar[r] & \pi_1 (\EuP^+_0(X)) \ar[r] \ar@{<->}[d]^{\rotatebox[origin=c]{90}{$\sim$}} & \pi_1([\EuP^+_0(X)/\aut^+ H^*(X)]) \ar[r] \ar@{<->}[d]^{\rotatebox[origin=c]{90}{$\sim$}} & \aut^+ H^*(X,\Z) \ar[r] \ar@{=}[d] & 1 \\
1 \ar[r] & \Auteq^0 \Db(X) \ar[r] & \Auteq \Db(X) \ar[r]^-{\varpi} & \aut^+ H^*(X,\Z)\ar[r] & 1.}\]
\end{cor}
\begin{proof}
Follows from the Galois correspondence, using the fact that $\pi$ is $\Auteq \Db(X)$-equivariant.
\end{proof}

\begin{cor}\label{cor:BBlkah}
If $\rho(X)=1$, then there is an isomorphism of short exact sequences
\[
\xymatrix{ 1 \ar[r] & \pi_1 (\EuQ^+_0(X)) \ar[r] \ar@{<->}[d]^{\rotatebox[origin=c]{90}{$\sim$}} & \pi_1(\cL_{\text{K\"ah}}(X)) \ar[r] \ar@{<->}[d]^{\rotatebox[origin=c]{90}{$\sim$}} & \Gamma^+(X) \ar[r] \ar@{=}[d] & 1 \\
1 \ar[r] & \Auteq^0 \Db(X) \ar[r] & \Auteq_{CY} \Db(X) \ar[r]^-{\varpi_K} &  \Gamma^+(X) \ar[r] & 1.}\]
\end{cor}
\begin{proof}
Follows from Corollary \ref{cor:BBaut}, together with the fact that $\EuQ^+_0(X) \subset \EuP^+_0(X)$ is a deformation retract (cf. Lemma \ref{lem:PQdefret}).
\end{proof}

\begin{cor}
\label{cor:BBses}
If $\rho(X) = 1$, then there is an isomorphism of short exact sequences
\begin{equation}
\label{eqn:dbcohses}
 \xymatrix{ 1 \ar[r] & \pi_1(\Omega^+_0(X)) \ar[r] \ar@{<->}[d]^{\rotatebox[origin=c]{90}{$\sim$}} & \pi_1(\cM_{\text{K\"ah}}(X)) \ar[r] \ar@{<->}[d]^{\rotatebox[origin=c]{90}{$\sim$}} & \Gamma^+(X) \ar[r] \ar@{=}[d] & 1 \\
1 \ar[r] & \Auteq^0 \Db(X)/[2] \ar[r] & \Auteq_{CY} \Db(X)/[2] \ar[r]^-{\varpi_K} & \Gamma^+(X) \ar[r] & 1.}
\end{equation}
\end{cor}
\begin{proof}
Recall that $\EuQ^+_0(X)$ is a $\C^*$-bundle over $\Omega^+_0(X)$. 
The image of the composition 
\[ \Z = \pi_1(\C^*) \to \pi_1(\EuQ^+_0(X)) \cong \Auteq^0 \Db(X)\]
consists of the even shifts (cf. \cite[Section 1]{BB}). 
The result now follows from Corollary \ref{cor:BBlkah} by taking a quotient.
\end{proof}

\section{Homological mirror symmetry}

A version of homological mirror symmetry was proved for certain $K3$ surfaces in \cite{SS}. 
In this section we recall the precise statement, and give some immediate formal consequences of it. 
These formal consequences are not specific to the setting of mirror symmetry for $K3$ surfaces: they should work the same for general compact Calabi--Yau mirror varieties. 
We will prove our main results in Section \ref{sec:symp_cons} by combining these formal consequences with geometric input specific to $K3$ surfaces.

\subsection{Homological mirror symmetry statement\label{Subsec:HMS}}

Let $\Lambda$ denote the universal Novikov field over $\C$:
\begin{equation} \Lambda \coloneqq  \left\{ \sum_{j=0}^\infty c_j \cdot q^{\lambda_j}: c_j \in \C, \lambda_j \in \R, \lim_{j \to \infty} \lambda_j = +\infty\right\}.\end{equation}
It is an algebraically closed field extension of $\C$, with a non-Archimedean valuation
\begin{align} 
\val: \Lambda & \to \R \cup \{\infty\} \\
\label{eqn:valdef} \val\left(\sum_{j=0}^\infty c_j \cdot q^{\lambda_j}\right) &\coloneqq  \min_j\{\lambda_j: c_j \neq 0\}.
\end{align}
If $c_1(X) = 0$, then the Fukaya category $\EuF(X,\omega)$ is a $\Lambda$-linear $\Z$-graded (non-curved) $A_{\infty}$-category whose objects are \emph{Lagrangian branes}: closed Lagrangian submanifolds $L \subset X$ equipped with orientations, gradings and Pin structures. 
A \emph{strictly unobstructed Lagrangian brane} is a pair $(L,J_L)$ where $L$ is a Lagrangian brane and $J_L$ an $\omega$-compatible almost-complex structure such that there are no non-constant $J_L$-holomorphic spheres intersecting $L$, nor non-constant $J_L$-holomorphic discs with boundary on $L$. 
For the purposes of this paper we define $\EuF(X,\omega)$ to be the Fukaya category of strictly unobstructed Lagrangian branes, since this can be constructed by classical means \cite{Seidel:HMSquartic,Seidel:Flux}. 
This restriction on the objects is harmless when $X$ has real dimension $\le 4$ (as is the case for $K3$ surfaces), because any Lagrangian brane $L$ admits a $J_L$ turning it into a strictly unobstructed Lagrangian brane by \cite[Lemma 8.4]{Seidel:HMSquartic}.\footnote{We recall the argument: somewhere-injective discs can be ruled out for generic $J_L$ by standard regularity and dimension-counting, and arbitrary discs are then ruled out by the theorem of Kwon--Oh \cite{KwonOh} or Lazzarini \cite{Lazzarini}, showing that any such disc contains a somewhere-injective disc in its image.}

Let $(X,\omega)$ be a symplectic $K3$, and $X^{\circ}$ a smooth Calabi--Yau algebraic surface over $\Lambda$. 
We say that $(X,\omega)$ is \emph{homologically mirror} to $X^{\circ}$ if there is a $\Z$-graded $\Lambda$-linear $A_\infty$ quasi-equivalence
\begin{equation}
\label{eqn:what_is_hms}
\Dpinf\EuF(X,\omega)\ \simeq \Dbdg(X^{\circ})
\end{equation}
where the left-hand side denotes the split-closure of the category of twisted complexes on the Fukaya category (i.e., $\Dpinf \EuC \coloneqq \Pi(Tw \,\EuC)$ in the notation of \cite{Seidel:FCPLT}), and the right-hand side denotes a $dg$ enhancement of the bounded derived category of coherent sheaves on $X^{\circ}$. 

We recall Batyrev's construction of mirror Calabi--Yau hypersurfaces in toric varieties \cite{Batyrev1993}, in the two-dimensional case. 
Fix polar dual reflexive three-dimensional lattice polytopes $\pol$ and $\pol\!^\circ$.  Let $\Xi_0$ denote the set of boundary lattice points of $\pol\!^\circ$ which do not lie in the interior of a codimension-one facet.    
Fix a vector $\lambda \in (\R_{>0})^{\Xi_0}$. 
On the $A$-side, we will consider hypersurfaces in resolutions of the toric variety associated to $\pol$.
The elements of $\Xi_0$ index the toric divisors which intersect such a hypersurface, and the coefficients $\lambda_\kappa$ of $\lambda$ will determine the cohomology class of the K\"ahler form on the hypersurface.

The vector $\lambda$ determines a function $\psi_{\lambda}: \pol\!^\circ \to \R$ as the smallest convex function such that $\psi_\lambda(0) = 0$ and $\psi_{\lambda}(\kappa) \ge -\lambda_\kappa$. We assume: 

\begin{addmargin}{2em} $(\ast)$ \emph{The decomposition of $\pol\!^\circ$ into domains of linearity of $\psi_{\lambda}$ is a simplicial refinement $\Sigma$ of the normal fan $\bar{\Sigma}$ to $\pol$, whose rays are generated by the elements of $\Xi_0$.}\end{addmargin}

\begin{rmk} 
In the language of \cite[Section 6.2.3]{CoxKatz}, condition $(\ast)$ holds if and only if $\lambda$ lies in the interior of the cone $\text{cpl}(\Sigma)$ of the secondary fan (or Gelfand--Kapranov--Zelevinskij decomposition) of $\Xi_0 \subset \R^3$, where $\Sigma$ is a simplified projective subdivision of $\bar{\Sigma}$.
\end{rmk}

 The morphism of fans $\Sigma \to \bar{\Sigma}$ induces a birational morphism of the corresponding toric varieties $Y \to \bar Y$, and $Y$ is an orbifold.  We will consider a smooth Calabi--Yau hypersurface $X \subset Y$ which avoids the orbifold points, and which is the proper transform of $\bar X \subset \bar Y$, a hypersurface with Newton polytope $\pol$.  There is a toric $\R$-Cartier divisor $D_{\lambda}=\sum_\kappa \lambda_\kappa \cdot D_\kappa$ with support function $\psi_{\lambda}$, which determines a toric K\"ahler form on $Y$, which restricts to a K\"ahler form $\omega_{\lambda}$ on $X$. 
 Its cohomology class $[\omega_\lambda]$ is the restriction of $PD(D_\lambda)$ to $X$.
 The pair $(X,\omega_{\lambda})$ depends up to symplectomorphism only on $(\pol,\lambda)$.

On the $B$-side, the polytope $\pol\!^\circ$ defines a line bundle over the toric variety $\bar Y^{\circ}$, whose global sections have a basis indexed by the lattice points of $\pol\!^\circ$.\footnote{In general, this is a toric stack, and one should consider the stacky derived category of hypersurfaces in this stack; in the situations relevant to our applications, the hypersurfaces will be smooth as schemes, and the stacky derived category coincides with the derived category of the underlying scheme.}
For any $d \in \Lambda^{\Xi_0}$ we have a corresponding hypersurface
\begin{equation} \label{eqn:B-side}
X^{\circ}_d \coloneqq \left\{\,  -\chi^0 + \sum_{\kappa\in\Xi_0} d_{\kappa}\cdot \chi^{\kappa} = 0\, \right\} \ \subset \ \bar Y^{\circ}.
\end{equation}

\begin{thm} \label{Thm:HMS}
Let $\pol$ and $\pol\!^\circ$ be polar dual reflexive three-dimensional simplices, and let $\lambda \in (\R_{>0})^{\Xi_0}$ satisfy $(\ast)$.
Then there exists a $d\in \Lambda^{\Xi_0}$, with $\val(d) = \lambda$, such that $(X,\omega_\lambda)$ is homologically mirror to $X^\circ_d$.
\end{thm}

\begin{proof} This is \cite[Theorem C]{SS}. 
We remark that the ``embeddedness'' and ``no $\bc$'' conditions of \emph{op. cit.} are automatic in this case. 
\end{proof}

\begin{rmk}
The fact that $\val(d) = \lambda$, with $\lambda$ satisfying $(\ast)$, implies that $X^{\circ}_d$ is smooth, see \cite[Proposition 4.4]{SS}. 
One can describe the ``mirror map'' $\lambda \mapsto d(\lambda)$ more precisely, but we shall not need the more precise version in this paper.  
\end{rmk}

\begin{rmk}
Theorem \ref{Thm:HMS} should continue to hold without the assumption that $\Delta$ and $\Delta\!^\circ$ are simplices, but it has not been proved.
\end{rmk}

\subsection{Examples revisited}
We want to use Theorem \ref{Thm:HMS} to find examples of $K3$s $(X,\omega)$ which are homologically mirror to an algebraic $K3$ $X^\circ$ of Picard rank $1$, so that we may apply the Bayer--Bridgeland theorem to compute the derived autoequivalences of $X^\circ$. 
Recall that the moduli space of $K3$ surfaces is 20-dimensional, and the subspace of algebraic $K3$ surfaces forms a countable union of hypersurfaces. 

Recall that $\Pic_{tor}(X) \subset H^2(X,\Z)$ is the sublattice spanned by the intersections of $X$ with the toric boundary divisors of $Y$. 
Theorem \ref{Thm:HMS} proves homological mirror symmetry for K\"ahler classes lying in an open subset of $\Pic_{tor}(X)_\R$, matching them up with a mirror family of algebraic $K3$ surfaces. 
According to Conjecture \ref{conj:lattmirr}, we expect members of the mirror family to satisfy
\[ \N(X^\circ) \cong \bL(X,\omega) \supseteq \Pic_{tor}(X)^\perp,\]
with equality holding if and only if $\omega$ is ambient-irrational. 
In particular their Picard rank is $\ge 20 - \rk (\Pic_{tor}(X))$, so can be equal to $1$ only if $\rk(\Pic_{tor}(X))$ is equal to $19$ or $20$. 

One might hope to find an example where $\Pic_{tor}(X)$ has rank $20$, and the mirror family is an open subset of the moduli space of $K3$ surfaces which could potentially intersect infinitely many of the codimension-$1$ families of $K3$ surfaces of Picard rank $\ge 1$. 
However this can't happen, because by construction our mirror families consist entirely of algebraic $K3$s, so must be contained in one of the codimension-$1$ loci of such.\footnote{One might hope to prove a homological mirror symmetry statement involving the bounded derived category of analytic coherent sheaves on a non-algebraic $K3$ surface, but in general this is a poorly-behaved invariant (not saturated, doesn't admit a generator\ldots). Thus, any meaningful statement of homological mirror symmetry for such $K3$ surfaces would require a significant revision of the definitions of the categories involved.} 
Thus the best one can hope for is an example where $\Pic_{tor}(X)$ has rank $19$, in which case we expect any ambient-irrational $\omega$ to be mirror to a $K3$ of Picard rank $1$. There are precisely two such examples arising from the Greene--Plesser construction, namely the mirror quartics and mirror double planes considered in Section \ref{sec:eg_fam}:

\begin{example}\label{eg:1}
Let $\pol$ be the reflexive 3-simplex with  vertices $\{(1,0,0), (0,1,0), (0,0,1), (-1,-1,-1)\}$, and $\pol\!^\circ$ its polar dual, which is given by the convex hull $\mathrm{Conv}(0,4e_1,4e_2,4e_3)-(1,1,1)$, so with vertices $\{(3,-1,-1), (-1,3,-1), (-1,-1,3), (-1,-1,-1)\}$.  
Then $\bar{Y}^\circ= \bP^3$ is smooth, and the anticanonical hypersurface $X^{\circ}_d$ is a quartic surface over $\Lambda$. 
The toric variety associated to $\pol$ is 
\[
\bar{Y} = [\bP^3 / \Pi] \cong \{x_0^4 - x_1x_2x_3x_4 = 0\} \subset \bP^4
\]
where $\Pi = (\Z/4)^2$ is as in Example \ref{Ex:mirror_quartic}. Anticanonical hypersurfaces in $\bar{Y}$ are hyperplane sections; these include the $\Pi$-invariant quartics  $Q_{\lambda} \subset \bP^3$ considered in \eqref{eqn:Dwork}, and in particular the Fermat quartic $Q_0$. $X$ is the crepant resolution of $Q_0 / \Pi$, resolving the 6 $A_3$-singularities. 
It supports a $19$-dimensional family of ambient K\"ahler forms. 
Theorem \ref{Thm:HMS} gives a quasi-equivalence
\[
\Dpinf\EuF(X,\omega) \simeq \Dbdg(X^{\circ}_{d}); \quad X^{\circ}_d \subset \bP^3 \, \textrm{a quartic}
\]  for ambient K\"ahler forms $\omega$ satisfying $(\ast)$. 
The valuations of the coefficients $d_\kappa$ in $19$ of the $20$ quartic monomials defining $X^{\circ}_d$ are determined by the $\omega$-areas of the exceptional resolution curves in $X$ and by its total volume, and the coefficient of the 20th monomial (corresponding to the interior point of $\Delta$) is $-1$.
\end{example}

\begin{example}\label{eg:2}
Let $\pol$ be the reflexive 3-simplex with vertices $\{(1,0,0), (0,1,0), (0,0,1), (-1,-1,-3)\}$, and $\pol\!^\circ$ its polar dual, which is given by the convex hull $\mathrm{Conv}(0,6e_1,6e_2,2e_3)-(1,1,1)$, so with vertices $\{ (-1,-1,-1), (5,-1,-1), (-1,5,-1), (-1,-1,1)\}$.  Then the toric variety $\bar{Y}^\circ = \bP(1,1,1,3)$, and the toric variety $\bar{Y}$ is an orbifold quotient $[\bP(1,1,1,3)/ \Pi']$ with $\Pi'$ as in Example \ref{Ex:mirror_double_plane}.   In the notation of that example, let $X$ denote the crepant resolution of $P_0 / \Pi'$, a hypersurface in the toric resolution of $[\bP(1,1,1,3)/ \Pi']$.  This supports a 19-dimensional family of ambient K\"ahler forms.  Theorem \ref{Thm:HMS} gives a quasi-equivalence
\[
\Dpinf\EuF(X,\omega)  \simeq \Dbdg(X^{\circ}_d); \quad X^{\circ}_d \subset \bP(1,1,1,3) \, \textrm{ a double plane}
\]
for ambient K\"ahler forms $\omega$ satisfying $(\ast)$. 
Here $X^{\circ}_d \to \bP^2 \supset \Sigma_d$ is the double branched cover of $\bP^2$ branched over a smooth sextic curve $\Sigma_{d}$ whose defining equation depends via the mirror map on $\omega$.  
\end{example}

\begin{rmk}\label{rmk:cubic_fourfold}
In \cite[Section 1.7.2]{SS} we elaborate a version of homological mirror symmetry similar to Theorem \ref{Thm:HMS}, taking the form
\begin{equation} \label{eqn:kuznetsov}
\Dpinf\EuF(X_{20},\omega)^\bc \simeq \EuA_{d(\omega)}.
\end{equation} 
It involves a $K3$ surface $X_{20}$ which is a complete intersection inside a toric variety, and has $\Pic_{tor}(X)$ of rank $20$. 
As we explained above, the mirror to the resulting $20$-dimensional family of K\"ahler forms could not be a 20-dimensional family of algebraic $K3$ surfaces; rather, the mirror is Kuznetsov's K3-category $\EuA_d$ of a cubic four-fold \cite{Kuznetsov}. 
This is equivalent to the derived category of certain algebraic $K3$ surfaces along certain 19-dimensional loci \cite{Addington-Thomas} (see also \cite{Huybrechts:K3category}). 
These loci should be mirror to certain rational hyperplanes in $\Pic_{tor}(X)$, so one should be able to prove that $(X_{20},\omega)$ is mirror to a $K3$ of Picard rank $1$, for $[\omega]$ an irrational point on such a hyperplane; these would then give many new examples of symplectic $K3$ surfaces to which our main results (e.g., Theorem \ref{Thm:Torelli}) apply. Indeed, whereas Examples \ref{eg:1} and \ref{eg:2} have $\N(X,\omega) \cong U \oplus \langle d \rangle$ for $d=2,4$, this construction would yield examples for all $d$ in the (infinite) set of natural numbers satisfying \cite[Condition (**)]{Huybrechts:K3category}.
\end{rmk}

\subsection{$A_\infty$ autoequivalences and Hochschild homology}
\label{subsec:ainfauteq}

The results of Section \ref{sec:derautk3} concern autoequivalences of triangulated categories, however our homological mirror equivalence \eqref{eqn:what_is_hms} concerns $A_\infty$ categories. 
Therefore we will develop the relevant theory for $A_\infty$ categories, before giving the relationship with triangulated categories in Section \ref{sec:enh}.

\begin{defn}
Let $\EuD$ be an $A_\infty$ category. We define $\Auteq(\EuD)$ to be the group of $A_\infty$ functors $F: \EuD \dashrightarrow \EuD$, such that $H^0(F)$ is an equivalence, considered up to isomorphism in $H^0(nu\hbox{-}fun(\EuD,\EuD))$, where $nu\hbox{-}fun(\EuD,\EuD)$ denotes the category of non-unital $A_{\infty}$-functors (see \cite[Section 1d]{Seidel:FCPLT}).
\end{defn}

The \emph{Hochschild homology} of an $A_\infty$ category $\EuD$ is  a graded vector space $HH_*(\EuD)$ (see \cite[Section 3.4]{Sheridan2015a}). 
There is a natural action of $\Auteq \EuD$ on $HH_*(\EuD)$ by graded automorphisms.
Thus, a homological mirror equivalence \eqref{eqn:what_is_hms} induces an isomorphism of exact sequences
\begin{equation}
\label{eqn:hms_ses_morph}
\xymatrix{
1 \ar[r] & \Auteq^0 \Dpinf\EuF(X) \ar@{<->}[d]^{\rotatebox[origin=c]{90}{$\sim$}} \ar[r] & \Auteq \Dpinf\EuF(X) \ar@{<->}[d]^{\rotatebox[origin=c]{90}{$\sim$}} \ar[r] & \Aut HH_*(\Dpinf\EuF(X)) \ar@{<->}[d]^{\rotatebox[origin=c]{90}{$\sim$}} \\
1 \ar[r] & \Auteq^0 \Dbdg(X^\circ) \ar[r] & \Auteq \Dbdg(X^\circ) \ar[r] & \Aut HH_*(\Dbdg(X^\circ)),}
\end{equation}
where $\Auteq^0$ is the subgroup consisting of autoequivalences acting trivially on $HH_*$.

For the Fukaya category we have the \emph{open--closed map}
\begin{equation}
\label{eqn:oc}
 \EuO\EuC: HH_*(\Dpinf\EuF(X)) \to H^{*+n}(X,\Lambda),
 \end{equation}
which under certain hypotheses is an isomorphism. 
For example, this is a formal consequence of the existence of a homological mirror $X^\circ$ which is \emph{maximally unipotent} (see \cite[Theorem 5.2]{Ganatra2015}) or \emph{smooth} (see \cite[Corollary 7]{Ganatra2016}), and in particular holds in the context of Theorem \ref{Thm:HMS}. 

For the derived category, \cite{Keller1998a} provides an isomorphism
\begin{equation}
\label{eqn:HHdbdg_hh}
  HH_*(\Dbdg(X^\circ)) \cong HH_*(X^\circ)
  \end{equation}
where the right-hand side denotes the Hochschild homology of the variety $X^\circ$ (see \cite{Caldararu:I}). 
On the other hand, we have the Hochschild--Kostant--Rosenberg isomorphism  
\begin{equation} \label{eqn:HKR}
HH_*(X^{\circ}) \simeq H^*(\Omega^{-*}X^\circ)
\end{equation}
(see, for instance, \cite{Caldararu:II,Huybrechts:FMAG}). 
Thus $HH_*(\EuD(X^\circ)) \simeq H^*(\Omega^{-*}X^\circ)$.

\subsection{Calabi--Yau autoequivalences}

\begin{defn}
An $n$-Calabi--Yau ($n$-CY) structure\footnote{There are various notions of Calabi--Yau structures, relevant in different contexts; the one used here was called a `weak proper Calabi--Yau structure of dimension $n$' in \cite[Section 6]{Ganatra2015}.} on a $\fk$-linear $A_{\infty}$-category $\EuD$  is a map $\phi: HH_{n}(\EuD) \to \fk$ with the property that the pairing
\begin{equation}\label{eq:serre_ainf}
\xymatrix{
\Hom^{*}(K,L) \otimes \Hom^{n-*}(L,K) \ar[r]^-{[\mu^2]} & \Hom^n(L,L) \to HH_n(\EuD) \ar[r]^-{\phi} & \fk
}
\end{equation}
is non-degenerate, for any objects $K,L$. 
\end{defn}

If $\EuD$ comes equipped with an $n$-CY structure $\phi$, we define $\Auteq_{CY} \EuD \subset \Auteq \EuD$ to be the subgroup fixing $\phi$. 
In general it may depend on the choice of $\phi$, but this is not the case for the categories involved in a homological mirror equivalence \eqref{eqn:what_is_hms}. 
These categories admit $2$-CY structures, and have $HH_2$ is of rank $1$ (being isomorphic to $H^4(X,\Lambda)\simeq H^2(\mathcal{O}_{X^\circ})$), so all $2$-CY structures are proportional. 
Thus, $\Auteq_{CY}$ does not depend on the choice of $2$-CY structure: if an autoequivalence preserves one it preserves all of them.

It is clear that $\Auteq^0 \subset \Auteq_{CY}$, so a homological mirror equivalence \eqref{eqn:what_is_hms} induces an isomorphism of exact sequences
\begin{equation}
\label{eqn:H0sesdgen}
\xymatrix{ 1 \ar[r] & \Auteq^0 \EuD\EuF(X) \ar@{<->}[d]^{\rotatebox[origin=c]{90}{$\sim$}} \ar[r] & \Auteq_{CY} \EuD\EuF(X) \ar@{<->}[d]^{\rotatebox[origin=c]{90}{$\sim$}} \ar[r] & \Aut HH_*(\EuD\EuF(X)) \ar@{<->}[d]^{\rotatebox[origin=c]{90}{$\sim$}}\\
1 \ar[r] & \Auteq^0 \Dbdg(X^\circ) \ar[r] & \Auteq_{CY} \Dbdg(X^\circ) \ar[r] & \Aut HH_*(X^\circ).}
\end{equation}

\subsection{Enhancements and autoequivalences}\label{sec:enh}

We would now like to relate the $A_\infty$ autoequivalence groups considered in the preceding sections with the triangulated autoequivalence groups considered by Bayer and Bridgeland. 
The relationship between $A_\infty$ and triangulated categories is given in \cite[Proposition 3.14]{Seidel:FCPLT}, which implies that if $\EuD$ is a triangulated $A_\infty$ category, then the cohomological category $H^0(\EuD)$ has a natural triangulated structure, and there is a homomorphism $\Auteq \EuD \to \Auteq H^0(\EuD)$ to the triangulated autoequivalence group, sending $F \mapsto H^0(F)$. 
We say that $\EuD$ is an $A_\infty$ enhancement of $H^0(\EuD)$. 

In our setting, $\Db(X^\circ)$ has a unique $dg$ (in particular, $A_\infty$) enhancement $\Dbdg(X^\circ)$ by \cite{Lunts2010}, and the results of \cite{Toen2006a,Lunts2016} 
imply that $\Auteq \Dbdg(X^\circ) \to \Auteq \Db(X^\circ)$ is in fact an isomorphism.

The homomorphism $\Auteq \Dbdg \to \Aut HH_*(\Dbdg)$ does not have an analogue in the general setting of triangulated categories; in the particular geometric setting that we consider, however, there is a substitute.
Any autoequivalence of $\Db(X^\circ)$ is of Fourier--Mukai type, and the Fourier--Mukai kernel induces an action on $HH_*(X^\circ)$ (see \cite{Caldararu:I}), so we obtain a map $\Auteq \Db(X^\circ) \to \Aut HH_*(X^\circ)$. 
This is identified with the map $\Auteq \Dbdg(X^\circ) \to \Aut HH_*(\Dbdg(X^\circ))$ by \cite[Theorem 2]{Ramadoss2010}. 

Since $\Auteq_{CY} \Db(X^\circ)$ consists precisely of the autoequivalences acting trivially on $HH_2(X^\circ)$ by \cite[Appendix A]{BB}, it gets identified with $\Auteq_{CY} \Dbdg(X^\circ)$. 
The upshot is an isomorphism of exact sequences
\begin{equation}
\label{eqn:H0sesdg}
\xymatrix{ 1 \ar[r] & \Auteq^0 \Dbdg(X^\circ) \ar@{<->}[d]^{\rotatebox[origin=c]{90}{$\sim$}} \ar[r] & \Auteq_{CY} \Dbdg(X^\circ) \ar@{<->}[d]^{\rotatebox[origin=c]{90}{$\sim$}} \ar[r] & \Aut HH_*(\Dbdg(X^\circ)) \ar@{<->}[d]^{\rotatebox[origin=c]{90}{$\sim$}}_{\eqref{eqn:HHdbdg_hh}}\\
1 \ar[r] & \Auteq^0 \Db(X^\circ) \ar[r] & \Auteq_{CY} \Db(X^\circ) \ar[r] & \Aut HH_*(X^\circ).}
\end{equation}

\subsection{The symplectic mapping class group acts on the Fukaya category}\label{sec:sympgracts}

Now recall \cite{Seidel:graded} that the group of graded symplectomorphisms $\Symp^{gr}(X)$ is a central extension by $\Z$ of the group of all symplectomorphisms $\Symp(X)$ of $X$. 
The Hamiltonian subgroup $\mathrm{Ham}^{gr}(X)$ is the connected component of $\Symp^{gr}(X)$ containing the identity, where we use the ``Hamiltonian topology'' on $\Symp(X)$ (see \cite[Remark 0.1]{Seidel:LecturesDehn} for the definition: when $H^1(X;\R) = 0$, there is no distinction between symplectic and Hamiltonian isotopy, so the Hamiltonian topology coincides with the $C^\infty$ topology). 
In particular we have an isomorphism $\Symp^{gr}(X) / \mathrm{Ham}^{gr}(X) \cong \pi_0\Symp^{gr}(X)$.

\begin{lem} There is a natural homomorphism
\begin{equation}
\label{eqn:sympactsFuk}
\pi_0 \Symp^{gr}(X) \longrightarrow \Auteq \Dpinf\EuF(X).
\end{equation}
\end{lem}

\begin{proof} The construction of a homomorphism $\Symp^{gr}/\mathrm{Ham}^{gr} \to \Auteq \EuF$ is explained in the setting of exact manifolds in \cite[Section 10c]{Seidel:FCPLT}, and the proof carries over to the strictly unobstructed setting in which we work. 
We compose this homomorphism with the homomorphism
\[ \Auteq \EuF \to \Auteq \Dpinf\EuF,\]
which exists for any $A_\infty$ category $\EuF$. 
\end{proof}

We now recall that $\EuF(X)$ admits a canonical $n$-CY structure $\phi$ (where $n$ is half the real dimension of $X$). 
It is defined by
\begin{equation} \label{eqn:weak_cy_fukaya}
\phi(\alpha) \coloneqq \int_X \mathcal{OC}(\alpha)
\end{equation}
where $\mathcal{OC}: HH_{n}(\EuF(X)) \to H^{2n}(X;\Lambda)$ is the open--closed string map  (compare \cite[Section 2.8]{Sheridan:Fano}). 
This induces an $n$-CY structure on $\Dpinf\EuF(X)$, by Morita invariance of Hochschild homology.

\begin{lem}
The map \eqref{eqn:sympactsFuk} lands in $\Auteq_{CY} \Dpinf\EuF(X)$, the subgroup of Calabi--Yau autoequivalences.
\end{lem}

\begin{proof}
Let $\psi$ be a graded symplectomorphism, inducing an autoequivalence $\Psi$ on $\EuF(X)$. We have
\begin{align*}
\int_X \mathcal{OC}(\Psi_* \alpha) &= \int_X \left(\psi^{-1}\right)^* \mathcal{OC}(\alpha) \quad \text{(naturality of $\mathcal{OC}$)} \\
&= \int_X \mathcal{OC}(\alpha) \quad \text{($\psi$ is orientation-preserving),}
\end{align*}
so $\Psi$ preserves $\phi$ as required.
\end{proof}

We denote the element of $\Symp^{gr}(X)$ corresponding to $k \in \Z$ by $[k]$. 
That is because the homomorphism \eqref{eqn:sympactsFuk} sends $[k]$ to the shift functor $[k]$. 
The central extension $\Symp^{gr}(X)/[2]$ of $\Symp(X)$ by $\Z/2$ is canonically split (see, e.g., \cite[Section B.5]{Sheridan2016}), so we obtain a homomorphism
\begin{equation}
\label{eqn:sympmcgacts}
 \Symp(X)/\mathrm{Ham}(X) \to  \Auteq_{CY}\Dpinf\EuF(X)/[2].
\end{equation}

\begin{cor} If the open--closed map $\mathcal{OC}: HH_*(\EuD\EuF(X)) \to H^{*+n}(X;\Lambda)$ is an isomorphism, then there is a natural homomorphism of exact sequences
\begin{equation}
\label{eqn:sessympmcgacts}
\xymatrix{1 \ar[r] & I(X,\omega) \ar[d] \ar[r] & G(X,\omega) \ar[r] \ar[d] & \Aut H^*(X,\Z) \ar@{^(->}[d]  \\
1 \ar[r] & \Auteq^0\Dpinf\EuF(X,\omega)/[2] \ar[r] & \Auteq_{CY}\Dpinf\EuF(X,\omega)/[2] \ar[r] & \Aut HH_*(\Dpinf\EuF(X,\omega)).}
\end{equation}
Here the rightmost vertical arrow is obtained by identifying $HH_*(\EuD\EuF) \cong H^{*+n}(X,\Z) \otimes \Lambda$ via $\mathcal{OC}$. Commutativity follows from the naturality of $\mathcal{OC}$.
\end{cor}

\subsection{Noncommutative Chern character}
\label{subsec:nc_chern}

If $\EuD$ is triangulated, we define its Grothendieck group $K(\EuD)$ to be equal to the Grothendieck group of the triangulated category $H^0(\EuD)$. 
The \emph{noncommutative Chern character} \cite[Definition 5.13]{Sheridan2015a} is a homomorphism
\[ \Ch: K(\Dbdg) \to HH_0(\Dbdg).\]

\begin{lem}\label{lem:oc_geom}
In the case of the Fukaya category, the composition
\[ \xymatrix{K(\Dpinf\EuF(X)) \ar[r]^-\Ch &HH_0(\Dpinf\EuF(X)) \ar[r]^-{\EuO\EuC} &H^n(X;\Lambda)}\]
takes the class of a Lagrangian $[L]$ to $PD([L])$, the Poincar\'e dual of its homology class.
\end{lem}
\begin{proof}
This follows because we have arranged that $L$ bounds no non-constant holomorphic discs, and the constant discs sweep out the fundamental cycle of $L$ (cf. \cite[Section 5a]{Seidel:biased}, \cite[Lemma 2.14]{Sheridan:Fano}). 
\end{proof}

In the case of the derived category, we introduce the isomorphism
\begin{equation}\label{eq:idhh}
 HH_*(\EuD(X^\circ)) \xrightarrow{\eqref{eqn:HHdbdg_hh}} HH_*(X^\circ) \xrightarrow{\eqref{eqn:HKR}} H^*( \Omega^{-*}{X^\circ}) \xrightarrow{\sqrt{\mathrm{td}(X^\circ)}\wedge -} H^*(\Omega^{-*}{X^\circ}).
 \end{equation}
Then the composition 
\[ K(\Dbdg(X^\circ)) \xrightarrow{\Ch} HH_0(\Dbdg(X^\circ)) \xrightarrow{\eqref{eq:idhh}}\bigoplus H^p\left(\Omega^p_{X^\circ}\right)\]
sends $[E]$ to its Mukai vector $v(E)$ \cite{Caldararu:II}. (If we had omitted the final twist by the square root of the Todd class in \eqref{eq:idhh} it would have sent $[E]$ to its Chern character.)

\begin{rmk}\label{rmk:MStell}
If $X^\circ$ is a complex $K3$ surface, then the diagram
\[\xymatrix{ \Auteq \EuD(X^\circ) \ar[rr] \ar@{<->}[d]^{\rotatebox[origin=c]{90}{$\sim$}} && HH_*(\EuD(X^\circ)) \ar@{<->}[d]^{\rotatebox[origin=c]{90}{$\sim$}}_-{\eqref{eq:idhh}} \\
\Auteq \Db(X^\circ) \ar[r]^-{\omega_\C} & \aut H^*(X,\C) \ar@{<->}[r]^-{\sim} & \aut H^*(\Omega^{-*} X^\circ)}\]
commutes, by \cite[Theorem 1.2]{Macri2009}. This would not have been the case had we not twisted by the square root of the Todd class in \eqref{eq:idhh}.
\end{rmk}

\section{$K3$ surfaces over the Novikov field} 

A $K3$ surface over an arbitrary field $\fk$ is an algebraic surface $X$ over $\fk$ with $H^1(X,\mathcal{O}_X) = 0$ and $K_X \cong \mathcal{O}_X$. 
 This again has a Picard group scheme $\Pic(X)$; we let $\NS(X)$ denote the Neron--Severi group, which is the quotient $\Pic(X) / \Pic^0(X)$ of the Picard group by its identity component. The fact that $H^1(X,\mathcal{O}_X) = 0$ implies that $\Pic(X)$ consists of rigid isolated points, so $\Pic(X) = \NS(X)$.

\subsection{Lefschetz principle\label{Sec:Lefschetz}}

We will use the ``Lefschetz principle'' (see e.g. \cite{Eklof}) to translate results about complex $K3$ surfaces into corresponding results about $K3$ surfaces over an arbitrary field of characteristic zero (of course, we have the Novikov field $\Lambda$ in mind). 
The results we collect here are taken from \cite{Huybrechts:K3book}, see also \cite[Proposition 5.4]{Huybrechts:Chow}, and will be standard in the relevant community. We include a discussion since the results may be less familiar to symplectic topologists.

Let $\fK$ be a field of characteristic zero, and $X$ a $K3$ surface over $\fK$. 
We can define $X$ using only a finite number of elements of $\fK$, so there exists a finitely-generated field $\Q \subset \fk_0 \subset \fK$ and a variety $X_0$ over $\fk_0$ such that $X = X_0 \times_{\fk_0} \fK$. 
We can embed $\fk_0$ in $\C$, so we obtain a variety $X_\C = X_0 \times_{\fk_0} \C$ (which depends on the choice of embedding). 
Using the fact that flat base change commutes with coherent cohomology, one can show that $X_0$, and therefore $X_\C$, are also $K3$ surfaces.  
We call $X_\C$ a ``complex model'' of $X$; the basic idea of the Lefschetz principle is to translate results about the complex $K3$ surface  $X_\C$ into results about $X$.

\begin{lem} \label{lem:picardiso}
Let $X$ be a $K3$ surface over an algebraically closed field $\fk$, $\fK/\fk$ a field extension, and $X_\fK := X \times_\fk \fK$.
Then the pull-back map $\Pic(X) \rightarrow \Pic(X_\fK)$ is a bijection.
\end{lem}

\begin{proof}[Sketch] This is \cite[Chapter 17, Lemma 2.2]{Huybrechts:K3book}.  To prove injectivity of the pull-back map, we observe that a line bundle $L \in \Pic(X)$ is trivial if and only if the composition map
\[ \Hom(L,\mathcal{O}_X)  \otimes \Hom(\mathcal{O}_X, L) \to \Hom(\mathcal{O}_X,\mathcal{O}_X) \cong \fk\]
is non-zero. 
Since coherent cohomology commutes with flat base change, this holds for $L$ if and only if it holds for the pull-back of $L$ to $X_\fK$ (note that this is true even if $\fk$ is not algebraically closed).

Surjectivity relies on the fact that $\fk$ is algebraically closed. 
Any line bundle on $X_K$ can be defined using finitely many elements of $\fK$, hence is defined over  some finitely-generated extension of $\fk$, which is the quotient field of a finitely-generated $\fk$-algebra $A$. 
Localizing $A$ with respect to finitely many denominators, we may assume that the line bundle can in fact be defined over $A$, so it can be viewed as a family of line bundles on $X$ parametrized by $\Spec (A)$. 
This family is classified by a morphism $\Spec(A) \to \Pic(X)$. 
The Picard scheme $\Pic(X)$ is reduced and has dimension equal to the rank of $H^1(X,\mathcal{O}_X)$, which for a $K3$ surface is $0$.
Since $\fk$ is algebraically closed, it follows that $\Pic(X)$ is simply a disjoint union of points: so the classifying morphism must be constant, which means that the line bundle is pulled back from $X$ as required.
\end{proof}

\begin{rmk}
A crucial point in the proof of surjectivity in Lemma \ref{lem:picardiso} was that line bundles on a $K3$ surface are \emph{rigid}: they do not admit non-trivial deformations.
\end{rmk}

\begin{cor}\label{cor:picard}
Let $X$ be a $K3$ surface over an algebraically closed field of characteristic zero, and $X_\C$ a complex model. 
Then $\Pic(X) \cong \Pic(X_\C)$.
\end{cor}
\begin{proof}
We observe that the algebraic closure $\bar{\fk}_0$ of $\fk_0$ embeds into $\C$ and $\fK$, so the restriction maps 
\[ \Pic(X) \to \Pic(X_0 \times_{\fk_0} \bar{\fk}_0) \leftarrow \Pic(X_\C)\]
are isomorphisms by Lemma \ref{lem:picardiso}.
\end{proof}

\begin{rmk}
It follows immediately that any $K3$ surface over an algebraically closed field of characteristic zero has Picard rank $\le 20$, since this is true of complex $K3$ surfaces. 
This is not true for $K3$ surfaces in finite characteristic.
\end{rmk}

\begin{rmk}
If $X_\C$ and $X'_\C$ are different complex models for $X$, then Corollary \ref{cor:picard} shows that they have isomorphic Picard lattices. However the embeddings $\Pic(X_\C) \hookrightarrow H^2(X_\C,\Z)$ and $\Pic(X'_\C) \hookrightarrow H^2(X'_\C,\Z)$ need not be isomorphic.
\end{rmk}

The following two results are discussed in \cite[Chapter 16, Section 4.2]{Huybrechts:K3book}:

\begin{lem}\label{lem:lef_spheres}
Let $X$ be a $K3$ surface over an algebraically closed field of characteristic zero, and $X_\C$ a complex model. 
Then the set of isomorphism classes of spherical objects of $\Db(X)$ is in bijection with the set of isomorphism classes of spherical objects of $\Db(X_\C)$.
\end{lem}
\begin{proof}
The proof follows that of Corollary \ref{cor:picard}, using the fact that spherical objects $\EuE$ are rigid because $\Ext^1(\EuE,\EuE) = 0$ by definition.
\end{proof}

\begin{lem}
\label{lem:lef_auteqs}
Let $X$ be a $K3$ surface over an algebraically closed field of characteristic zero, and $X_\C$ a complex model. 
Then there is an isomorphism $\Auteq \Db(X) \cong \Auteq \Db(X_\C)$, inducing isomorphisms:
\begin{equation}
\label{eqn:basechangeses}
\xymatrix{
1 \ar[r] & \Auteq^0\Db(X^\circ_\C)/[2] \ar[r] \ar@{<->}[d]^{\rotatebox[origin=c]{90}{$\sim$}} & \Auteq_{CY}\Db(X^\circ_\C)/[2] \ar[r] \ar@{<->}[d]^{\rotatebox[origin=c]{90}{$\sim$}} & \Aut HH_*(X^\circ_\C) \ar@{<-->}[d] \\
1 \ar[r] & \Auteq^0\Db(X^\circ)/[2] \ar[r] & \Auteq_{CY}\Db(X^\circ)/[2] \ar[r] & \Aut HH_*(X^\circ). }
\end{equation}
The dashed vertical arrow signifies an isomorphism between the images of the rightmost horizontal arrows. 
In other words, this diagram can be completed to an isomorphism of short exact sequences, by replacing the rightmost terms by the images of the rightmost horizontal arrows.
\end{lem}

\begin{proof}
The proof follows that of Corollary \ref{cor:picard}, using the fact that the Fourier--Mukai kernels defining autoequivalences are rigid (because $H^1(X,\mathcal{O}_X) = 0$). 
The action of autoequivalences on Hochschild homology is compatible with flat base change, so these isomorphisms preserve $\Auteq^0$ (the subgroup acting trivially on $HH_*$) and $\Auteq_{CY}$ (the subgroup acting trivially on $HH_2$). 
The dashed vertical arrow is induced by the inclusions
\[ HH_*(X_\C) \hookleftarrow HH_*(X_0 \times_{\fk_0} \bar{\fk}_0) \hookrightarrow HH_*(X)\]
which are induced by base change by the inclusions $\C \hookleftarrow \bar{\fk}_0 \hookrightarrow \fK$.
\end{proof}

\begin{rmk} The previous result does not hold for general varieties, since if $H^1(X;\mathcal{O}_X)$ is non-zero then one can tensor by (flat) line bundles with different structure group after extending scalars; the result relies on the fact that the Picard group is discrete.
\end{rmk}   

Now we will consider point-like objects $\EuE$ of $\Db(X)$, i.e. objects satisfying $\Ext^*(\EuE,\EuE) \cong \wedge^* (\fK^{\oplus 2})$. 
Given such an $\EuE$, we may choose a finitely-generated $\Q \subset \fk_0 \subset \fK$ such that $X$ and the object $\EuE$ are defined over $\fk_0$ (indeed, for any finite set of objects we can choose a finitely-generated field $\fk_0$ over which they are all defined). 
Thus we have an object $\EuE_0$ of $\Db(X_0 \times_{\fk_0} \bar{\fk}_0)$ which pulls back to $\EuE$ on $X$ and to $\EuE_\C$ on the complex model $X_\C$. 
Since coherent cohomology commutes with flat base change, $\EuE_0$ and $\EuE_\C$ are also point-like.

\begin{lem}
\label{lem:pointsnonvan}
Let $X$ be a $K3$ surface over an algebraically closed field of characteristic zero, with $\rho(X) = 1$. Then any point-like object $\EuE$ of $\Db(X)$ has non-zero Chern character $\Ch(\EuE) \neq 0 \in HH_0(X)$.
\end{lem} 

\begin{proof}
First we prove the result assuming $\fK = \C$. 
Bayer and Bridgeland show that any point-like (or spherical) object $\EuE$ of $\Db(X)$ is quasi-stable for some stability condition, meaning that it is semistable and all its stable factors have positively proportional Mukai vector. Indeed, for a stability condition $\sigma$, let the  $\sigma$-width of a point-like object be the difference between the phases of the maximal and minimal semistable factors in its Harder--Narasimhan filtration. Then \cite[Lemma 6.3]{BB} shows that there exists a stability condition $\sigma$ such that $\EuE$ has $\sigma$-width $0$, and \cite[Proposition 3.15]{BB} shows that either $\EuE$ is $\sigma$-quasi-stable, or it is $\sigma'$-quasi-stable for a stability condition $\sigma'$ near $\sigma$.
A quasi-stable object has non-trivial Mukai vector $v(\EuE) \neq 0 \in \N(X)$, and therefore non-trivial Chern character $\Ch(\EuE) \neq 0 \in HH_0(X)$ as required.

Now we address the general case. 
Let $\EuE$ be a point-like object of $\Db(X)$, and let us choose a complex model $X_\C$ for $X$ over which $\EuE$ is defined. 
Taking Hochschild homology and Chern characters commute with flat base change, so $\EuE$ has non-zero Chern character if and only if $\EuE_\C$ does: since $\rho(X_\C) = \rho(X) = 1$ by Corollary \ref{cor:picard}, the result follows from the case $\fK=\C$ proved above.
\end{proof}

\begin{lem}
\label{lem:pointsprim}
Let $X$ be a $K3$ surface over an algebraically closed field of characteristic zero, with $\Pic(X) \cong \langle 2n \rangle$ where $n$ is square-free. 
Then for any point-like object $\EuE$ of $\Db(X)$, there exists a spherical object $S$ with $\chi(\EuE,S) = 1$.
\end{lem}

\begin{proof}
First we prove the result for $\fK = \C$. 
Let $\EuE$ be a point-like object of $\Db(X)$: we will start by showing that there exists a $K3$ surface $Y$ and an equivalence $\eta: \Db(X) \xrightarrow{\sim} \Db(Y)$ taking $\EuE$ to the skyscraper sheaf of a point. 
The set of stability conditions making $\EuE$ quasi-stable is open \cite[Proposition 2.10]{BB}, and non-empty since $\EuE$ is point-like, cf. the proof of Lemma \ref{lem:nonzero_in_homology}. Thus we can pick a stability condition $\sigma$ which makes $\EuE$ quasi-stable and which is generic in the sense of \cite{Bayer-Macri}.  Suppose the Mukai vector of $\EuE$ is $v = m \cdot v_0$ where $v_0$ is primitive and $m \in \bZ_+$.  The Mukai pairing $(v_0,v_0)=0$, since $\chi(\EuE,\EuE) = 0$ using the fact that $\EuE$ is point-like.  Now, in \cite[Section 6 and Proof of Lemma 7.2]{Bayer-Macri}, Bayer and Macr\`i show that there is a non-empty projective moduli stack of $\sigma$-semistable objects with the same Mukai vector as $\EuE$. Let $\scrM(v)$ be its coarse moduli space; then the coarse moduli space of $\scrM(v_0)$ is again a $K3$ surface, which has a distinguished Brauer class $\alpha$.  There is a Fourier--Mukai equivalence
\[
\eta: \Db(X) \xrightarrow {\sim}\Db(\scrM(v_0), \alpha)
\]
(where the right-hand side denotes the derived category of twisted sheaves), which takes any complex in $\Db(X)$ defining a point of $\scrM(v)$ to a torsion sheaf on the twisted $K3$ surface $\scrM(v_0)$ of dimension $0$ and length $m$. \cite{Bayer-Macri} proves that $\eta$ identifies $\scrM(v)$ with the $m$-th symmetric product $\Sym^m(\scrM(v_0))$. It follows that the general point of $\scrM(v)$ corresponds, under this identification, to the skyscraper sheaf of an $m$-tuple of pairwise distinct points, which has $\Ext^1$ of rank $2m$. Since $\Ext^1$ varies upper semicontinuously,  it follows that $\Ext^1(\EuE,\EuE)$ has rank at least $2m$; since $\EuE$ is point-like, $m=1$ and $v=v_0$.  The derived equivalence $\eta$ then takes $\EuE$ to the skyscraper sheaf of a point on the twisted $K3$ surface $(\scrM(v), \alpha)$. 

Recall that a $K3$ surface $Y$ equipped with a Brauer class $\beta$ is called a \emph{twisted Fourier--Mukai partner} of $X$ if there is an equivalence $\Db(X) \simeq \Db(Y,\beta)$: so $(Y,\beta) \coloneqq (\scrM(v),\alpha)$ is a twisted Fourier--Mukai partner of $X$. 
Ma gives a formula for the number of twisted Fourier--Mukai partners of $X$ with Brauer class of a given order \cite[Proposition 5.1]{MaShouhei}. 
His result shows that if $\Pic(X) \cong \langle 2n \rangle$ with $n$ square-free, then any twisted Fourier--Mukai partner of $X$ has trivial Brauer class; since this holds by hypothesis in our case, we must have $\beta=0$.
Now $\mathcal{O}_Y$ is a spherical object of $\Db(Y)$ with $\chi(\mathcal{O}_y,\mathcal{O}_Y)=1$ for any $y \in Y$, so if we set  $S\coloneqq \eta^{-1}(\mathcal{O}_{Y})$ then $\chi(\EuE,S) = 1$ as required.

Now we address the general case. 
Let $\EuE$ be a point-like object, and let us choose a complex model $X_\C$ for $X$ over which $\EuE$ is defined. 
Then $\EuE_\C$ is a point-like object of $\Db(X_\C)$, and $\Pic(X_\C) \cong \Pic(X)$ by Corollary \ref{cor:picard}, so there exists a spherical object $S_\C$ of $\Db(X_\C)$ for which $\chi(\EuE_\C,S_\C)=1$ by the previous argument. 
The spherical object descends to $X_0 \times_{\fk_0} \bar{\fk}_0$ by Lemma \ref{lem:lef_spheres}, so we obtain a spherical object $S$ of $\Db(X)$ with $\chi(\EuE,S) = 1$ as required.  
\end{proof}

\subsection{Obtaining a Picard rank one mirror}

Let $K$ be an algebraically closed field of characteristic zero.  Let $p: \EuX \rightarrow \cM$ be a family of $K3$ surfaces over $K$, i.e. a proper smooth morphism of relative dimension $2$ with both the relative dualising sheaf and $R^1p_*\mathcal{O}_{\EuX}$ being trivial.  

\begin{thm}\label{thm:MP}
Let $\rho_0$ be the Picard rank of the generic fibre $\EuX_{\eta}$. The locus $\{t \in \cM \, | \, \rho(\EuX_t) > \rho_0 \}$ (called the \emph{Noether--Lefschetz locus}) is a countable union of positive-codimension algebraic subvarieties.
\end{thm}

\begin{proof} This is classical; references which explicitly deal with general algebraically closed fields include \cite{Ciliberto_et_al} and \cite{Maulik-Poonen}.
\end{proof}

Now we consider the family $\EuX^\circ \xrightarrow{d} \cM = \bA^{\Xi_0}$ of hypersurfaces in $\bar Y^\circ$ that is defined by \eqref{eqn:B-side} (it is defined over $\Z$, hence over any field). 
Suppose that $\bar Y^\circ$ is smooth as a scheme, away from its toric fixed points, so that the generic fibre of $\EuX^\circ$ is a smooth $K3$ surface. 
Suppose furthermore that, after base changing to $\C$, the generic fibre has Picard rank $\rho_0$.

\begin{prop}\label{prop:picard_rk_one_exists}
There exists a set $\Upsilon \subset \R^{\Xi_0}$, a countable union of hyperplanes of rational slope passing through the origin, such that if $d \in \Lambda^{\Xi_0}$ has valuation $\val(d) \notin \Upsilon$, then $X^{\circ}_d$ is smooth with Picard rank $\rho_0$.
\end{prop} 

\begin{proof}
Let $\Delta \subset \cM$ denote the discriminant locus of the family $\EuX^\circ$. 
By Lemma \ref{lem:picardiso}, the generic fibre of the family after base changing to $\Lambda$ has Picard rank $\rho_0$. 
 By Theorem \ref{thm:MP}, the hypersurfaces of higher Picard rank are contained in a countable family of algebraic hypersurfaces in $\Lambda^{\Xi_0} \setminus \Delta$.
 
We now restrict the family to $(\Lambda^*)^{\Xi_0} \setminus \Delta$. 
The valuation image of any algebraic hypersurface in $(\Lambda^*)^{\Xi_0}$ is a polyhedral complex of dimension $|\Xi_0|-1$ called the \emph{tropical amoeba} (see, e.g., \cite[Proposition 3.1.6]{Maclagan2007}). 
In particular it is contained in a finite union of \emph{affine} hyperplanes of rational slope. 

Now we observe that, for any $a \in \R_{>0}$, the map $q \mapsto q^a$ extends to an automorphism $\psi_a$ of $\Lambda$. 
We define a corresponding automorphism $\Psi_a = (\psi_a,\ldots,\psi_a)$ of $\Lambda^{\Xi_0}$, which satisfies $\val(\Psi_a(d)) = a \cdot \val(d)$ and $X^\circ_{\Psi_a(d)} \cong \psi_a^* X^\circ_d$. 
It follows that the discriminant and Noether--Lefschetz loci are invariant under $\Psi_a$, hence that their tropical amoebae are invariant under scaling by $a$. 
Since $a \in \R_{>0}$ was arbitrary, it follows that the affine hyperplanes making up the tropical amoebae pass through the origin. 
Therefore we can take $\Upsilon$ to be the union of linear hyperplanes containing the tropical amoebae of the discriminant and Noether--Lefschetz loci. \end{proof}

\begin{prop}
\label{prop:irratmeansgen}
In the situation of Theorem \ref{Thm:HMS}, suppose that $\lambda \in (\R_{>0})^{\Xi_0}$ is irrational, i.e. does not lie on any rational hyperplane. 
Then  $X^{\circ}_{d}$ has minimal Picard rank $\rho_0$.
\end{prop}

\begin{proof}
This follows from Proposition \ref{prop:picard_rk_one_exists}: 
since the hyperplanes making up $\Upsilon$ have rational slope, they cannot contain $\lambda = \val(d)$. \end{proof}

\section{Symplectic consequences}
\label{sec:symp_cons}

Let $(X,\omega)$ be a symplectic $K3$, and $X^\circ$ an algebraic $K3$ surface over $\Lambda$. 
Our standing assumptions for this section are as follows. 
We assume that $(X,\omega)$ is homologically mirror to $X^\circ$:
\begin{equation}
\label{eqn:HMSass}
\Dpinf\EuF(X,\omega) \simeq \Dbdg(X^{\circ});
\end{equation}
and furthermore that they satisfy the analogue of Conjecture \ref{conj:lattmirr}:
\begin{equation}
\label{eq:lattmirr}
 \bL(X,\omega) \cong \N(X^\circ) \coloneqq U \oplus \Pic(X^\circ).
 \end{equation}

\begin{example}
\label{eg:GPirrat}
These assumptions are satisfied if $X$ is a mirror quartic or mirror double plane, $\omega_\lambda$ is an ambient-irrational K\"ahler form, and $X^\circ = X^\circ_d$ is the Greene--Plesser mirror. 
Firstly, \eqref{eqn:HMSass} holds by Theorem \ref{Thm:HMS}. 
Ambient-irrationality of $\omega_\lambda$ allows us to verify \eqref{eq:lattmirr} as follows. 
Firstly, because $\omega_\lambda$ is ambient-irrational, we may choose $\lambda = \val(d)$ to be irrational, so that $\Pic(X^\circ_{d}) = \langle 2n \rangle$ by Proposition \ref{prop:irratmeansgen} (where $n=2$ for the quartic, $n=1$ for the double plane). 
Secondly, ambient-irrationality implies that $\bL(X,\omega_\lambda) \cong U \oplus \langle 2n \rangle$ (see Examples \ref{Ex:mirror_quartic} and \ref{Ex:mirror_double_plane}). Thus \eqref{eq:lattmirr} also holds.
\end{example}

For the purposes of this section we will pick a complex model $X^\circ_\C$ of $X^\circ$. 
We will abbreviate $\Po \coloneqq \Pic(X^\circ)$.

\subsection{Spherical objects}

The homological mirror equivalence \eqref{eqn:HMSass} gives us an isomorphism of numerical Grothendieck groups
\begin{equation}
\label{eq:knum_iso}
K_{num}(\EuD\EuF(X)) \simeq K_{num}(\EuD(X^\circ)).
\end{equation}
The Mukai vector defines an isomorphism
\begin{equation}
\label{eq:knum_muk}
K_{num}(\Dbdg(X^\circ))^- \simeq \N(X^\circ) = U \oplus \Po,
\end{equation}
see \cite[Section 16.2.4]{Huybrechts:K3book} and Remark \ref{rmk:Nrelevance} (recall the `$-$' denotes negation of the pairing). 
Composing, we obtain an isomorphism
\begin{equation}\label{eq:knum_fuk_id}
K_{num}(\EuD\EuF(X))^- \simeq U \oplus \Po.
\end{equation}

We now consider spherical objects in $\EuD\EuF(X)$. Their classes in the numerical Grothendieck group correspond, under \eqref{eq:knum_fuk_id}, to elements of $\Delta(U \oplus \Po)$. One example of a spherical object is a Lagrangian sphere, and one example of a Lagrangian sphere is a vanishing cycle. 
Our aim in this subsection is to prove the following:

\begin{lem}\label{lem:sph_cl_vc}
Any class in $\Delta(U \oplus \Po) $ is represented by a vanishing cycle in $X$.
\end{lem}

Before proving Lemma \ref{lem:sph_cl_vc} we will prove some preliminary results, the first of which concerns the composition
\begin{equation}\label{eq:OC_ch}
K(\EuD\EuF(X)) \xrightarrow{\Ch} HH_0(\EuD\EuF(X)) \xrightarrow{\mathcal{OC}} H^n(X,\Lambda).
\end{equation}

\begin{lem}
The map \eqref{eq:OC_ch} descends to an injective map
\begin{equation}\label{eq:OC_knum}
K_{num}(\EuD\EuF(X)) \hookrightarrow H^n(X,\Lambda).
\end{equation}
\end{lem}
\begin{proof}
Because the Chern character map $\Ch: K(\EuD(X^\circ)) \to HH_0(\EuD(X^\circ))$ descends to an injection $K_{num}(\EuD(X^\circ)) \hookrightarrow HH_0(\EuD(X^\circ))$, the same holds true for $\EuD\EuF(X)$ by our homological mirror symmetry assumption. 
Composing with $\mathcal{OC}$, which is an isomorphism in this case, gives the result.
\end{proof}

We now define $\mathcal{S} \subset K_{num}(\EuD\EuF(X))$ to be the set of classes represented by vanishing cycles; $\mathbb{S} \subset K_{num}(\EuD\EuF(X))$ the subgroup they generate; and $\mathbb{T} \subset U \oplus \Po$ the subgroup generated by $\Delta(U \oplus \Po)$. 

\begin{lem}\label{lem:S_OC}
The map \eqref{eq:OC_knum} identifies $\mathcal{S}$ with $\Delta(\bL(X,\omega))$, and $\mathbb{S}^-$ isometrically with the subgroup of $\bL(X,\omega) \subset H^n(X,\Lambda)$  generated by $\Delta(\bL(X,\omega))$.
\end{lem}
\begin{proof}
The map \eqref{eq:OC_knum} sends a Lagrangian sphere to its homology class in $\bL(X,\omega)$, by Lemma \ref{lem:oc_geom}. 
This homology class must lie in $\Delta(\bL(X,\omega))$, and in fact all such classes are realized by vanishing cycles, by Remark \ref{rmk:dehntwist}. 
The identification respects pairings, by Remark \ref{rmk:Lrelevance}. 
\end{proof}

\begin{lem}\label{lem:S_T}
The isomorphism \eqref{eq:knum_fuk_id} identifies $\mathbb{S}$ with $\mathbb{T}$.
\end{lem}
\begin{proof}
We start by observing that the subgroup $\mathbb{T} \subset U \oplus \Po$ has full rank by an elementary argument \cite[Lemma A.4]{SS:k3cob}. 
In particular, the pairing restricted to it is nondegenerate, so $\mathbb{T}$ is a sub\emph{lattice}.

Now the isomorphism \eqref{eq:knum_fuk_id} clearly identifies $\mathbb{S}$ with a subgroup of $\mathbb{T}$. 
On the other hand, Lemma \ref{lem:S_OC} identifies $\mathbb{S}^-$ with the subgroup of $\bL(X,\omega)$ generated by $\Delta(\bL(X,\omega))$, which the isomorphism \eqref{eq:lattmirr} identifies with $\mathbb{T}^-$. 
The result now follows because it is impossible to embed the lattice $\mathbb{T}$ properly inside itself, by discriminant considerations.
\end{proof}

\begin{proof}[Proof of Lemma \ref{lem:sph_cl_vc}] Follows by combining Lemmas \ref{lem:S_OC} and \ref{lem:S_T}.
\end{proof}

\begin{rmk}\label{rmk:knum_not_quite}
Of course one expects \eqref{eq:knum_fuk_id} to identify $K_{num}(\EuD\EuF(X))^-$ with $\bL(X,\omega)$, however we have not proved this. 
The issue is that a general object of $\EuD\EuF(X)$ is an idempotent summand of a twisted complex, and while Lemma \ref{lem:oc_geom} tells us what the image under \eqref{eq:knum_fuk_id} of the twisted complex is, it does not tell us anything about the image of the idempotent summand (e.g., whether it lies in $H^n(X,\Z)$).
\end{rmk}

\subsection{Symplectic mapping class groups via HMS}

Suppose $\rk(\Po) = 1$ (as in Example \ref{eg:GPirrat}).
We compose the various morphisms of short exact sequences of groups that we have constructed: the maps on the central terms are
\begin{multline}
\label{eqn:composeses}
\pi_1(\cM_0(U \oplus \Po)) \xrightarrow{\eqref{eqn:sesmonodromy2}}  G(X,\omega) \xrightarrow{\eqref{eqn:sessympmcgacts}}  \Auteq_{CY}\Dpinf\EuF(X,\omega)/[2] \simxleftrightarrow{\eqref{eqn:hms_ses_morph}} \Auteq_{CY} \Dbdg(X^\circ)/[2] \ldots \\
\ldots \simxleftrightarrow{\eqref{eqn:H0sesdg}} \Auteq_{CY} \Db(X^\circ)/[2] \simxleftrightarrow{\eqref{eqn:basechangeses}} \Auteq_{CY} \Db(X^\circ_\C)/[2] \simxleftrightarrow{\eqref{eqn:dbcohses}} \pi_1(\cM_0(U \oplus \Po)).
\end{multline}
Note that in some cases we wrote exact sequences without a `$\to 1$' on the right: we turn these into short exact sequences in the canonical way by replacing the rightmost group with the image of the rightmost homomorphism in the exact sequence, so that all of the morphisms in \eqref{eqn:composeses} are morphisms of short exact sequences.

We recall the meanings of these morphisms:
\begin{itemize}
\item \eqref{eqn:sesmonodromy2} comes from symplectic monodromy.
\item \eqref{eqn:sessympmcgacts} comes from the action of the symplectic mapping class group on the Fukaya category; it relies on the open--closed map being an isomorphism, which is a consequence of homological mirror symmetry.
\item \eqref{eqn:hms_ses_morph} exists by our assumption that $(X,\omega)$ and $X^\circ$ are homologically mirror. 
\item \eqref{eqn:H0sesdg} maps Calabi--Yau autoequivalences of $\Dbdg(X^\circ)$ to Calabi--Yau autoequivalences of $\Db(X^\circ)$ by taking $H^0$.
\item \eqref{eqn:basechangeses} identifies the derived autoequivalence group of a $K3$ over $\Lambda$ with that of a $K3$ over $\C$.
\item \eqref{eqn:dbcohses} exists  by the theorem of Bayer--Bridgeland, because $\Pic(X^\circ_\C) \cong \Po$ has rank $1$.  
We have also replaced $HH_*(X^\circ_\C)$ with $H^*(X^\circ_\C,\C)$ via \eqref{eq:idhh}, which is valid by Remark \ref{rmk:MStell}.
\end{itemize}
 
\begin{rmk}
Considering the right-most terms in our exact sequences, we observe that \eqref{eqn:sesmonodromy2} maps to $\Aut H^2(X,\Z)$ whereas \eqref{eqn:sessympmcgacts} maps from $\Aut H^*(X,\Z)$. 
In order to be able to turn the composition into a morphism of short exact sequences, we use the fact that any symplectomorphism acts trivially on $H^0(X,\Z) \oplus H^4(X,\Z)$.
 \end{rmk}
 
\begin{prop}
\label{prop:sescompat}
Suppose $\rk(\Po) = 1$. Then the composition of morphisms of short exact sequences \eqref{eqn:composeses} is an isomorphism.
\end{prop}
\begin{proof}
By the 5-lemma, it suffices to prove the result on the initial and final terms of the short exact sequences. 
The maps between final terms are all isomorphisms, so it remains to prove the result for the initial terms.
The composition of morphisms between initial terms has the form
\begin{multline} 
\label{eqn:comp_init}
\pi_1(\Omega^+_0(U \oplus \Po)) \to I(X,\omega) \to \Auteq^0 \Dpinf\EuF(X,\omega)/[2] \cong  \Auteq^0 \Dbdg(X^\circ)/[2] \ldots \\ \ldots \cong \Auteq^0\Db(X^\circ)/[2] \cong \Auteq^0 \Db(X^\circ_\C)/[2] \cong \pi_1(\Omega^+_0(U \oplus \Po)).
\end{multline}
Because $\rk(\Po)=1$, Section \ref{sec:pic1eg} gives an identification
\[ \Omega^+_0(U \oplus \Po) \cong \mathfrak{h} \setminus \bigcup_{\delta \in \Delta(U \oplus \Po)/\pm \id} \{p_{\delta}\}.\]
Therefore $\pi_1(\Omega^+_0(U \oplus \Po))$ is a free group with generators indexed by $\Delta(U \oplus \Po)/\pm \id$.
 
Choose paths from a basepoint in $\Omega^+_0(U \oplus \Po)$ to each point $p_\delta$: these specify generators for the free group. 
The monodromy around $p_\delta$ is $\tau_L^2 \in I(X,\omega)$, a squared Dehn twist in the corresponding vanishing cycle $L$ (cf. Remark \ref{rmk:dehntwist}).
The corresponding autoequivalence of $\Dpinf\EuF(X,\omega)$ is $\mathrm{Tw}^2_L$, the squared algebraic twist in the corresponding spherical object by \cite{Seidel:FCPLT}, which gets sent to the autoequivalence $\mathrm{Tw}^2_S$ for the corresponding spherical object $S$ of $\Db(X^\circ_\C)$. 

Corollary \ref{cor:BBses} says that $\Auteq^0 \Db(X^\circ_\C)/[2] \cong \pi_1(\Omega^+_0(U \oplus \Po))$.
By \cite[Proposition 2.3]{BB}, for every $\epsilon \in \Delta(U \oplus \Po)$ there exists a spherical vector bundle $S_\epsilon$, so that the spherical twist $\mathrm{Tw}^2_{S_\epsilon}$ corresponds to a loop enclosing the hole $p_\epsilon$ under this isomorphism.  
Furthermore, by \cite[Remark 6.10]{BB}, $\Auteq^0 \Db(X^\circ_\C)$ acts transitively on the set of spherical objects $S$ with fixed Mukai vector $v(S) = \epsilon$. 
It follows that \emph{every} squared spherical twist in a spherical object $S$ with Mukai vector $v(S) = \epsilon$ corresponds to a loop in $\pi_1(\Omega^+_0(U \oplus \Po))$ enclosing the hole $p_\epsilon$.

In particular, the composition \eqref{eqn:comp_init} sends a loop enclosing $p_\delta$ to a loop enclosing $p_\epsilon$, where $\delta$ is the image of $[L]$ under the map \eqref{eq:knum_fuk_id} and $\epsilon$ is the Mukai vector of $S$, and the spherical objects $L$ and $S$ correspond under mirror symmetry. 
The map $\delta \mapsto \epsilon$ is bijective by Lemma \ref{lem:sph_cl_vc}. 
Therefore the composition \eqref{eqn:comp_init} sends one set of generators for the free group bijectively to another set of generators for the free group, so the map is an isomorphism as required.
\end{proof}

\begin{rmk}
Let us briefly mention a minor generalization of Proposition \ref{prop:sescompat}: it continues to hold if we replace $\cM_0(U \oplus \Po)$ with $\cL_0(U \oplus \Po)$, the symplectic mapping class group $G(X,\omega)$ with the graded symplectic mapping class group $\pi_0\Symp^{gr}(X,\omega)$, and remove all of the quotients by even shifts in \eqref{eqn:composeses}. The only point that requires extra explanation is the construction of the symplectic monodromy map \eqref{eqn:sesmonodromy2}. The crucial points are that a choice of fibrewise holomorphic volume form on the family $\EuX \to B$ induces a lift of the symplectic monodromy map to the graded symplectic mapping class group, and the fibre of the $\C^*$-bundle $\cL_0(U \oplus \Po) \to \cM_0(U \oplus \Po)$ can be identified with the space of holomorphic volume forms on the corresponding fibre of the family $\EuX(U \oplus \Po,\kappa) \to \cM_0(U \oplus \Po)$. The details are left to the reader.
\end{rmk}

Now recall that
\[Z(X,\omega) \coloneqq  \ker \, \left[G(X,\omega) \to \Auteq\Dpinf\EuF(X,\omega)/[2] \right]\]
denotes the ``Floer-theoretically trivial'' subgroup of the symplectic mapping class group (it is contained in $I(X,\omega)$, when the open--closed map is an isomorphism). 
We have the following immediate corollary of Proposition \ref{prop:sescompat}:

\begin{cor}
\label{cor:floertrivialsemi}
Suppose $\rk(\Po) = 1$.
Then we have
\begin{align*}
G(X,\omega) & \cong Z(X,\omega)  \rtimes \pi_1(\cM_{\text{cpx}}(X,\omega))  \quad \text{and}\\
I(X,\omega) & \cong Z(X,\omega) \rtimes \pi_1(\Omega^+_0(X,\omega)) .
\end{align*}
\end{cor}

We also recall that
\[ K(X,\omega) \coloneqq  \ker \, \left[\pi_0\Symp(X,\omega) \to \pi_0\Diff(X)\right]\]
denotes the smoothly trivial symplectic mapping class group. 
It is contained in $I(X,\omega)$.

\begin{cor} \label{cor:smoothly_trivial_inf_generated}
Suppose $\rk(\Po) = 1$. 
Then we have homomorphisms
\begin{equation}
\label{eqn:smooth_comp_id}
 \pi_1(\Omega^+_0(X,\omega)) \to K(X,\omega) \to \pi_1(\Omega^+_0(X,\omega))
 \end{equation}
 whose composition is an isomorphism. 
In particular, $K(X,\omega)$ is infinitely generated.
\end{cor}
\begin{proof}
We consider the composition $\pi_1(\Omega^+_0(X,\omega)) \to I(X,\omega) \to \pi_1(\Omega^+_0(X,\omega))$ appearing in \eqref{eqn:comp_init}, which is shown to be an isomorphism in Proposition \ref{prop:sescompat}. 
Observe that the image of the map $\pi_1(\Omega^+_0(X,\omega)) \to I(X,\omega)$ is generated by squared Dehn twists, which are known to be smoothly trivial: so we can replace $I(X,\omega)$ with the subgroup $K(X,\omega)$. 
\end{proof}

Now we can refine Remark \ref{rmk:space_of_symp}, which we recall concerned the space $\Omega$ of symplectic forms on $X$ cohomologous to $\omega$. 
It essentially consisted of the observation that the composition
\begin{equation}
\label{eqn:conn_hom}
\pi_1(\Omega) \twoheadrightarrow K(X,\omega) \twoheadrightarrow \pi_1(\Omega^+_0(X,\omega))
\end{equation}
is surjective, where the first map is the connecting map in the long exact sequence associated to a Serre fibration, and the second was constructed in the proof of Corollary \ref{cor:smoothly_trivial_inf_generated}; this shows that $\pi_1(\Omega)$ is infinitely generated.

\begin{lem}
\label{lem:symp_forms}
There exists a map $i:\Omega^+_0(X,\omega) \to \Omega$, such that the composition
\begin{equation}
\label{eqn:comp_symp_forms}
 \pi_1(\Omega^+_0(X,\omega)) \xrightarrow{i_*} \pi_1(\Omega) \xrightarrow{\eqref{eqn:conn_hom}} \pi_1(\Omega^+_0(X,\omega))
 \end{equation}
is an isomorphism.
\end{lem}
\begin{proof}
Consider the family of marked K\"ahler $K3$ surfaces $\EuX(\bL(X,\omega),[\omega]) \to \Omega^+_0(X,\omega)$. 
We saw in the proof of Corollary \ref{cor:smoothly_trivial_inf_generated} that the monodromy homomorphisms are smoothly trivial; since $\Omega^+_0(X,\omega)$ is an Eilenberg--MacLane space, it follows that the family is smoothly trivial. 
Pulling back a choice of fibrewise K\"ahler form via a smooth trivialization induces a map $\Omega^+_0(X,\omega) \to \Omega$, and it is clear from the definitions that the composition 
\[ \pi_1(\Omega^+_0(X,\omega)) \xrightarrow{i_*} \pi_1(\Omega) \to K(X,\omega)\]
is precisely the symplectic monodromy homomorphism. 
It follows that the composition \eqref{eqn:comp_symp_forms} is equal to the composition \eqref{eqn:smooth_comp_id}, which is an isomorphism by Corollary \ref{cor:smoothly_trivial_inf_generated}.
\end{proof}

\subsection{Lagrangian spheres}

We now consider Lagrangian spheres $L \subset (X,\omega)$. 
We declare two Lagrangian spheres to be ``Fukaya-isomorphic'' if they can be equipped with a grading and Pin structure such that the corresponding objects of $\EuF(X,\omega)$ are quasi-isomorphic. 

\begin{lem}[= Theorem \ref{Thm:spheres} (1)]
\label{lem:lag_sphere_vc}
Suppose $\rk(\Po) = 1$. 
Then any Lagrangian sphere $L \subset (X,\omega)$ is Fukaya-isomorphic to a vanishing cycle.
\end{lem}
\begin{proof}
Let $L$ be a Lagrangian sphere, and $\EuE_L$ the mirror spherical object of $\Db(X^\circ_\C)$. 
By Lemma \ref{lem:sph_cl_vc}, there exists a vanishing cycle $V$ with $v(\EuE_V) = v(\EuE_L)$.  
We can apply \cite[Remark 6.10]{BB} (because $\rho(X^\circ_\C)=1$), which says that there exists $\Phi \in \Auteq^0 \Db(X^\circ_\C)$ with $\Phi(\EuE_V) \cong \EuE_L$.  
We have a homomorphism $\pi_1(\Omega^+_0(X,\omega)) \to \Auteq^0 \Db(X^\circ_\C)$ from the composition of short exact sequences \eqref{eqn:composeses}, which is surjective by Proposition \ref{prop:sescompat}. 
Therefore, there exists $a \in \pi_1(\Omega^+_0(X,\omega))$ such that $a \cdot V \cong L$. 
It is clear that $a \cdot V$ is the vanishing cycle whose vanishing path is the concatenation of the vanishing path of $V$ with the loop $a$, so we are done.
\end{proof}

\begin{lem}
\label{lem:orbitshom}
Suppose $\rk(\Po) = 1$. 
Then the orbits of $G(X,\omega)$ on the set of Fukaya-isomorphism classes of Lagrangian spheres in $(X,\omega)$ are in bijection with the orbits of $\Gamma^+(U \oplus \Po)$ on $\Delta(U \oplus \Po)/\pm \id$.
\end{lem}

\begin{proof} 
Recall that the image of $G(X,\omega) \to \aut H^2(X,\Z)$ is precisely $\Gamma^+(U \oplus \Po)$ (cf. Proposition \ref{prop:shrunkmon}). Thus there is a well-defined map
\begin{equation}
\label{eqn:Lhomclass}
 \{\text{Lagrangian spheres up to Fukaya-isomorphism}\}/G(X,\omega) \to (\Delta(U \oplus \Po)/\pm \id)/\Gamma^+(U \oplus \Po)
 \end{equation}
sending $L$ to its homology class (the choice of orientation of $L$ does not matter since we quotient by $\pm \id$). 
This map is surjective by Lemma \ref{lem:sph_cl_vc}.

Suppose that $L_1$ and $L_2$ have the same image under \eqref{eqn:Lhomclass}. 
By Lemma \ref{lem:lag_sphere_vc}, we may assume that $L_1$ and $L_2$ are vanishing cycles. 
Because $G(X,\omega)\to \Gamma^+(U \oplus \Po)$ is surjective, we may assume that $[L_1] = [L_2] = \delta$ in $H_2(X,\Z)$; so $L_1$ and $L_2$ are vanishing cycles from the same point $p_\delta$. 
It follows that their vanishing paths differ by an element of $\pi_1(\Omega^+_0(U \oplus \Po))$, so the vanishing cycles differ by its image under $\pi_1(\Omega^+_0(U \oplus \Po)) \to I(X,\omega) \subset G(X,\omega)$. 
In particular, $L_1$ and $L_2$ represent the same class on the left-hand side of \eqref{eqn:Lhomclass}, so the map is injective as required.
\end{proof}

\begin{cor}
If $\rk(\Po) = 1$, then $G(X,\omega)$ has finitely many orbits on the set of Fukaya-isomorphism classes of Lagrangian spheres.
\end{cor}
\begin{proof}
By Lemma \ref{lem:orbitshom}, it suffices to check that $\Gamma^+(U \oplus \Po)$ has finitely many orbits on $\Delta(U \oplus \Po)$. 
This follows from the following general result of Borel and Harish-Chandra \cite[Theorem 6.9]{BHC}.  Let $G$ be a reductive algebraic group over $\Q$, $V$ a rational finite-dimensional representation of $G$, $\Gamma \subset V_{\Q}$ a $G_{\Z}$-invariant lattice, and $Q$ a closed orbit of $G$. Then $Q\cap \Gamma$ consists of finitely many $G_{\Z}$-orbits.
\end{proof}

\begin{cor}[= Theorem \ref{Thm:spheres} (2)] \label{cor:unique_sphere}
If $\Po = \langle 2 \rangle$ or $\langle 4 \rangle$, then $G(X,\omega)$ acts transitively on the set of  Fukaya-isomorphism classes of Lagrangian spheres. 
\end{cor}

\begin{proof}
Follows by Lemma \ref{lem:orbitshom} and Lemma \ref{lem:oneorbit}.
\end{proof}

\begin{cor} \label{cor:all_conj}
Suppose $\Po = \langle 2 \rangle$ or $\langle 4 \rangle$.
 If $L$ and $L'$ are Lagrangian spheres in $X$, then the Dehn twists $\tau_L$ and $\tau_{L'}$ are conjugate in $\Auteq_{CY} \Dpinf\EuF(X)$.
\end{cor}

\begin{proof}
It is sufficient to prove that the algebraic twist functors $\mathrm{Tw}_L$ and $\mathrm{Tw}_{L'}$ are conjugate in $\Auteq_{CY}\Dpinf\EuF(X)$.  The twist functors are conjugate provided the underlying spherical objects lie in the same orbit of the autoequivalence group, by \cite[Lemma 5.6]{Seidel:FCPLT}. 
\end{proof}

The Dehn twist in a simple non-separating curve on a closed surface of genus $g\geq 2$ has non-trivial roots in the mapping class group \cite{Margalit-Schleimer}.  For $g \geq 2$, the image of a Dehn twist  in the symplectic group $Sp(2g,\Z)$ has a cube root (this fails when $g=1$).

\begin{prop}\label{prop:roots}
Suppose $\Po = \langle 2 \rangle$ or $\langle 4 \rangle$.
Let $L\subset X$ be a Lagrangian sphere.  The Dehn twist $\tau_L$ admits no cube root in $G(X,\omega)$.
\end{prop}

\begin{proof}
By Proposition \ref{prop:sescompat} and Lemma \ref{lem:free_prod}, we have a surjection $G(X,\omega) \twoheadrightarrow \pi_1(\cM_0(U \oplus \Po)) \cong \Z \ast \Z/p$. 
Furthermore the Dehn twist in a given vanishing cycle can be arranged to correspond to $1 \in \Z$ (cf. Remark \ref{rmk:dehntwist}). 
By Corollary \ref{cor:all_conj}, it follows that all Dehn twists map to elements conjugate to $1 \in \Z$. 
Now we pass to the abelianization $(\Z \ast \Z/p)^{ab} \cong \Z \oplus \Z/p$: all Dehn twists map to $(1,0)$, and therefore have no cube root.
\end{proof}

Note that since $\tau_L^2$ is smoothly trivial, $\tau_L$ is its own cube root in $\pi_0\mathrm{Diff}(X)$, so this is a symplectic and not smooth phenomenon.

\subsection{Lagrangian tori}

We now consider embedded Lagrangian tori $L \subset (X,\omega)$ with Maslov class zero. 

\begin{lem}\label{lem:primitive}
Suppose that $\Po = \langle 2n \rangle$ where $n$ is square-free (observe that this is true in the situation of Example \ref{eg:GPirrat}, where $n = 1$ or $2$). Then a Maslov-zero Lagrangian torus $L \subset (X,\omega)$ represents a primitive homology class.
\end{lem}

\begin{proof}
The object $L$ of $\Dpinf\EuF(X,\omega)$ corresponds to a point-like object $\EuE$ of $\Db(X^\circ)$. 
Therefore there exists a spherical object $S$ of $\Db(X^\circ)$ with $\chi(\EuE,S) = 1$, by Lemma \ref{lem:pointsprim}. 
By Lemma \ref{lem:sph_cl_vc} there exists a vanishing cycle $V$ whose class in the numerical Grothendieck group corresponds to that of $S$ under mirror symmetry, so we have
\[-[L] \cdot [V] = \chi(L,V) = \chi(\EuE,S) = 1.\]
It follows that $[L]$ is primitive. 
\end{proof}

\begin{rmk}
It would be nice to prove Lemma \ref{lem:primitive} by arguing that the Chern character of the mirror object $\EuE$ to $L$ is primitive, and that therefore the Chern character $[L]$ of $L$ is primitive. 
However this would require us to prove that the map \eqref{eq:knum_fuk_id} identifies $K_{num}(\EuD\EuF(X))$ with a primitive sublattice of $H^2(X,\Z)$, which we have not done (cf. Remark \ref{rmk:knum_not_quite}): that is why we circumvent this issue by appealing 
to Lemma \ref{lem:sph_cl_vc}.
\end{rmk}

For our final result, we change our assumptions on the symplectic $K3$ surface $(X,\omega)$. 
Rather than assuming that $(X,\omega)$ itself has a homological mirror satisfying certain hypotheses, we assume that $\omega$ is a \emph{limit of $\rho^\circ=1$ classes}. 
We define this to mean that there exists a smooth one-parameter family of K\"ahler forms $\omega(t)$ with $\omega(0) = \omega$, having the following property: there exist times $t_1,t_2,\ldots$ converging to $0$ such that for all $j$, $(X,\omega(t_j))$ admits a homological mirror $X^\circ_j$ of Picard rank $\rho(X^\circ_j) = 1$.
 
\begin{lem} \label{lem:nonzero_in_homology}
Suppose $\omega$ is a limit of $\rho^\circ=1$ classes. 
If $L \subset (X,\omega)$ is a Maslov-zero Lagrangian torus, then $[L] \neq 0 \in H_2(X,\Z)$.
\end{lem}

\begin{proof}
Suppose $L\subset (X,\omega)$ is a Maslov-zero Lagrangian torus with vanishing homology class.  A Moser-type argument shows that $L$ deforms to give a Maslov-zero Lagrangian torus $L(t) \subset (X,\omega(t))$ for $t \in [0,\epsilon)$ sufficiently small that $L(t)$ remains embedded. 
Thus there exists $\tau \in (0,\epsilon)$ such that $L(\tau) \subset (X,\omega(\tau))$ is embedded and $(X,\omega(\tau))$ admits a homological mirror $X^\circ$ of Picard rank $1$.  
$L(\tau)$ defines a point-like object of $\Dpinf\EuF(X,\omega(\tau))$, and we denote the mirror point-like object by $\EuE \in \Dbdg(X^{\circ})$.  Under the isomorphism
\[
H^2(X;\Lambda) \cong HH_0(\Dpinf\EuF(X,\omega(\tau))) \cong HH_0(\Dbdg(X^{\circ})) \cong HH_0(X^\circ)
\]
we see that $\EuE$ has trivial Chern character. 
However $X^\circ$ has Picard rank $1$, so this contradicts Lemma \ref{lem:pointsnonvan}.
\end{proof}

It is obvious that, if $(X,\omega)$ is mirror to a $K3$ surface of Picard rank $1$, then $\omega$ is a limit of $\rho^\circ=1$ classes. 
However there are more examples, as we now describe.
 
Recall that we constructed the mirror quartic and mirror double plane as hypersurfaces $X \subset Y$ depending on data $(\pol,\lambda)$. 
The vector $\lambda \in (\R_{>0})^{\Xi_0}$ determines a refinement $\Sigma$ of $\bar{\Sigma}$, and hence a toric resolution of singularities $Y \to \bar{Y}$; and $X \subset Y$ is the proper transform of $\bar{X} \subset \bar{Y}$. 
The vector $\lambda$ also determines a K\"ahler class $[\omega_\lambda]$ on $X$.

The $K3$ surface $X$ does not depend on the refinement $\Sigma$, so long as the rays of $\Sigma$ are generated by $\Xi_0$: changing $\Sigma$ corresponds to changing $Y$ by a birational modification in a region disjoint from $X$ (the same is not true in higher dimensions).
A convenient way of organizing the different refinements is provided by the ``secondary fan'' or ``Gelfand--Kapranov--Zelevinskij decomposition'' associated to $\Xi_0 \subset \R^3$, whose cones $\text{cpl}(\Sigma)$ are indexed by fans $\Sigma$ whose rays are generated by subsets of $\Xi_0$ (see \cite{OP} or \cite[Section 6.2.3]{CoxKatz}). 
Our prescription for determining $\Sigma$ from $\lambda$ is equivalent to defining $\Sigma$ to be the unique fan such that $\lambda$ is contained in the interior of $\text{cpl}(\Sigma)$.

For any $\lambda$ lying in the interior of a cone $\text{cpl}(\Sigma)$, where $\Sigma$ is a refinement of $\bar{\Sigma}$ whose rays are generated by $\Xi_0$, we obtain a K\"ahler class $\omega_\lambda$ on $X$. 
We call these K\"ahler classes the \emph{ambient} K\"ahler classes on $X$.
Theorem \ref{Thm:HMS} only provides a homological mirror to $(X,\omega_\lambda)$ if $\Sigma$ is furthermore \emph{simplicial} (the corresponding cones $\text{cpl}(\Sigma)$ are the top-dimensional ones).
Nevertheless we have:

 \begin{lem}
 \label{lem:kahler_limit}
Let $X$ be the mirror quartic or mirror double plane, and $\omega$ an ambient K\"aher class. 
Then $\omega$ is a limit of $\rho^\circ=1$ classes. 
\end{lem}
\begin{proof}
Suppose $\omega = \omega_\lambda$; then $\lambda$ lies in the closure of a top-dimensional cone $\text{cpl}(\Sigma')$, where $\Sigma'$ is a simplicial refinement of $\bar{\Sigma}$ whose rays are generated by $\Xi_0$. 
Therefore we can choose a path $\lambda(t)$ converging linearly to $\lambda(0) = \lambda$ along a line of irrational slope, such that $\lambda(t)$ lies in the interior of $\text{cpl}(\Sigma')$ for $t \neq 0$. 
Then Theorem \ref{Thm:HMS} provides a mirror $X^\circ_{d(t)}$ to $(X,\omega_{\lambda(t)})$ for all $t \neq 0$. 
Because $\lambda(t)$ moves linearly along a line of irrational slope, there is a sequence of times $t_j$ converging to $0$ at which $\lambda(t_j)$ is irrational, so that the mirror $X^\circ_{d(t_j)}$ has Picard rank $1$ by Proposition \ref{prop:irratmeansgen}.
\end{proof}

Lemmas \ref{lem:kahler_limit} and \ref{lem:nonzero_in_homology} combine to prove Theorem \ref{Thm:Lag_torus}.

\bibliographystyle{amsalpha}
\bibliography{mybib}

\end{document}